\documentclass[11pt,a4paper]{amsart}
\usepackage[utf8]{inputenc}
\usepackage{amsthm,amssymb,esint,bm,mathtools}
\usepackage[margin=1.2in]{geometry}
\usepackage{hyperref,hypcap}
\hypersetup{colorlinks=true,allcolors=blue,bookmarksdepth=3}
\usepackage{enumitem}

\usepackage{microtype}

\newtheorem{thm}{Theorem}[section]
\newtheorem{cor}[thm]{Corollary}
\newtheorem{lem}[thm]{Lemma}

\newtheorem{conj}{Conjecture}

\theoremstyle{definition}
\newtheorem{df}[thm]{Definition}
\theoremstyle{remark}
\newtheorem{rem}[thm]{Remark}

\numberwithin{equation}{section}

\newcommand{\dd}{\mathop{}\!\mathrm{d}}
\DeclareMathOperator{\diam}{diam}
\DeclareMathOperator{\supp}{supp}
\newcommand{\dist}{\mathrm{dist}}

\author[P.~Germain]{Pierre Germain}
\address[Pierre Germain]{Department of Mathematics, Huxley building, South Kensington campus, Imperial College London, London SW7 2AZ, United Kingdom}
\email{pgermain@ic.ac.uk}

\author[I.~Moyano]{Iván Moyano}
\address[Ivan Moyano]{Laboratoire J.A. Dieudonné,
UMR CNRS 7351, Université Côte-d'Azur, Parc Valrose 06108 Nice Cedex 02, France}
\email{imoyano@unice.fr}

\author[H.~Zhu]{Hui Zhu}
\address[Hui Zhu]{Department of Mathematics, Huxley building, South Kensington campus, Imperial College London, London SW7 2AZ, United Kingdom}
\email{hui.zhu@ic.ac.uk}

\title{On the vanishing of eigenfunctions of the Laplacian on tori}

\begin{document}

\begin{abstract} 
    Consider an eigenfunction of the Laplacian on a torus. How small can its $L^2$-norm be on small balls?
    We provide partial answers to this question by exploiting the distribution of integer points on spheres, basic properties of polynomials, and Nazarov--Turán type estimates for exponential polynomials. Applications to quantum limits and control theory are given.
\end{abstract}

\maketitle

\tableofcontents 

\section{Introduction}

\subsection{Eigenfunctions of the Laplacian on tori}

On the flat torus $\mathbb{T}^d = \mathbb{R}^d / (2\pi \mathbb{Z})^d$ of dimension $d \ge 2$, let $\mathcal{E}_d(\lambda) = \{ u \in C^\infty(\mathbb{T}^d) \mid -\Delta u = \lambda^2 u \}$ be the eigenspace of the Laplacian associated with the eigenvalue $\lambda^2 \in \mathbb{N}$ (where $\lambda \ge 0$).
Elements of $\mathcal{E}_d(\lambda)$ can be written as Fourier series
\begin{equation*}
    u(x) = (2\pi)^{-\frac{d}{2}} \sum_{k \in \mathcal{S}_d(\lambda)} \widehat{u}_k e^{i k \cdot x},
\end{equation*}
where $\mathcal{S}_d(\lambda) = \lambda\mathbb{S}^{d-1} \cap \mathbb{Z}^d$.
Denoting $\# A$ the cardinal of a set $A$, we set
\begin{equation*}
    \mathcal{N}_d(\lambda) = \# \mathcal{S}_d(\lambda) = \dim \mathcal{E}_d(\lambda).
\end{equation*}
Recall that (see e.g., \cite{Grosswald, IwaniecKowalski} and references therein):
\begin{itemize}
    \item $\mathcal{N}_2(\lambda) \lesssim \exp\{\frac{C \ln \lambda}{\ln \ln \lambda}\}$ for some $C>0$.
          Moreover $\mathcal{N}_2(\lambda) = 0$ if and only if the prime factorization of $\lambda^2$ includes a factor $p^k$ where $p \equiv 3 \pmod{4}$ and $k \equiv 1 \pmod{2}$.
    \item $\mathcal{N}_3(\lambda) \lesssim \lambda \ln\lambda \ln\ln\lambda$.
          Moreover $\mathcal{N}_3(\lambda) \gtrsim_\epsilon \lambda^{ 1 - \epsilon}$ for all $\epsilon > 0$ if $\lambda^2 \ne 0,4,7 \pmod{8}$, and $\mathcal{N}_3(\lambda) = 0$ if and only if $\lambda^2 \ne 4^a(8b+7) $ for some $a,b \in \mathbb{N}$.
    \item $0 < \mathcal{N}_4(\lambda) \lesssim \lambda^2 \ln\ln\lambda $.
        Moreover 
        \begin{equation*}
            \limsup\limits_{\lambda \to \infty} \frac{\mathcal{N}_4(\lambda)}{\lambda^2 \ln\ln \lambda} \gtrsim 1,
            \quad
            \liminf\limits_{\lambda\to\infty} \mathcal{N}_4(\lambda) \lesssim 1.
        \end{equation*}

    \item $\mathcal{N}_d(\lambda) \sim_d \lambda^{d-2}$ for all $d \ge 5$.
\end{itemize}
Readers are advised to refer to \S\ref{sec::notations} for notations used in the above estimates and throughout the paper.
To simplify statements, we will always assume that $\mathcal{E}_d(\lambda) \ne \emptyset$.
When $\lambda^2 \in \mathbb{N}$, this assumption holds for all $d \ge 4$, but requires further arithmetic conditions for $d=2,3$ as we have seen.

\subsection{Observability of eigenfunctions}

Denoting
\begin{equation*}
    \fint_\Omega f(x) \dd x = \frac{1}{|\Omega|} \int_\Omega f(x) \dd x
\end{equation*}
for any measurable $\Omega \subset \mathbb{R}^d$ with $|\Omega| > 0$ and $f \in L^1(\Omega)$, our aim in the present article is to understand the optimal \emph{observability constant} of eigenfunctions from $B_r = \{x \in \mathbb{R}^d : |x| < r\}$, a ball of radius $r>0$:
\begin{equation}
    \label{eq::def-m}
    m_d(\lambda,r) = \inf_{0 \ne u \in \mathcal{E}_d(\lambda)} \frac{\displaystyle \fint_{B_r} |u(x)|^2 \dd x}{\displaystyle \fint_{\mathbb{T}^d} |u{(x)}|^2 \dd x}.
\end{equation}
Clearly the definition of $m_d(\lambda,r)$ is independent of the choice of the center of the ball.
The behavior of $m_d(\lambda,r)$ as $r\to 0$ quantitatively describes the local vanishing (in the $L^2$ sense) of eigenfunctions.

\subsection{Known results}

An obvious remark is that $m_d(\lambda,r) \le 1$ for all $\lambda,r,d$.
To see this, it suffices to test the right hand side of \eqref{eq::def-m} with the eigenfunction defined by $u(x) = e^{i k \cdot x}$, where $k \in \mathcal{S}_d(\lambda)$.
Beyond this trivial example, the state of the art is as follows:
\begin{itemize}
    \item In all dimensions:
    \begin{itemize}
        \item Connes \cite{Connes} proved that $m_d(\lambda,r) \sim_{d,r} 1$ 
        \item Fontes-Merz \cite{Fontes-Merz} extended a Turán-type estimate of Nazarov \cite{Nazarov} to higher dimensions which implies that
        \begin{equation*}
            m_d(\lambda,r) \gtrsim (C_d r^d)^{2 d \lambda} = (\sqrt[d]{C_d} r)^{2 d^2 \lambda}.
        \end{equation*}
        Similar estimates can be derived from the work by Egidi and Veselic \cite{EgidiVeselic}.
    \end{itemize}
    \item When $d=2$:
    \begin{itemize}
        \item The resolvent estimate by Tao \cite{Tao} implies that $m_2(\lambda,r) \gtrsim r^6$ if $r > \lambda^{-\frac{1}{16}}$.
    \end{itemize}
\end{itemize}

\textit{Different control sets} have been considered in the literature.
Replacing $B_r$ with a piece of curved hypersurface $\mathcal{H}$, Bourgain and Rudnick \cite{BourgainRudnick} show that 
\begin{equation*}
    \inf\limits_{0 \ne u \in \mathcal{E}_d(\lambda)} \frac{\displaystyle \fint_{\mathcal{H}} |u(x)|^2 \dd x}{\displaystyle \fint_{\mathbb{T}^d} |u{(x)}|^2 \dd x} \sim_{d,\mathcal{H}} 1
\end{equation*}
if $d=2,3$, while Egidi and Veselić \cite{EgidiVeselic} consider ``thick'' sets in the spirit of the Logvinenko--Sereda theorem.
In our recent work with Burq and Sorella \cite{BurqGermainSorellaZhu}, we studied observability from rough control sets satisfying only certain Fourier decay conditions.

\textit{Different types of Fourier localization} are also of interest. The work of Fontes-Merz cited above applies to more general situations, for instance functions whose Fourier coefficients are supported in the ball --- instead on the sphere --- of radius $\lambda$. 
Being able to distinguish these two different situations is at the heart of the present article.

\textit{For generic eigenfunctions,} it is expected that some form of quantum ergodicity should hold. It was indeed proved by Lester and Rudnick \cite{LesterRudnick}, as well as by Granville and Wigman \cite{GranvilleWigman} that for not-too-small $r$ and for generic eigenfunctions, there holds $m_d(\lambda,r) \sim_d 1$.

\textit{For different geometries}, say on a general compact Riemannian manifold $M$, the behavior of the similarly defined function $m_M$ (see \S\ref{sec::other-geometries}) can change drastically.
Indeed, this dependence on the global geometry of the manifold makes the study of $m_M$ very appealing.
We will explore in \S\ref{sec::other-geometries} the behavior of $m_M$ on various geometries.

\subsection{Notations}
\label{sec::notations}

For parameters $a_1, \dots, a_n$, we will denote by $C_{a_1,\ldots,a_n}$, $C'_{a_1,\ldots,a_n}$ etc.\ for \emph{positive} constants whose values depend only on these parameters.
Universal constants independent of any parameters will simply be denoted by $C$, $C'$ etc.

For two quantities $A$ and $B$, we write $A \lesssim_{a_1, \dots, a_n} B$ if $A \le C_{a_1,\ldots,a_n} B$ for some $C_{a_1,\ldots,a_n} > 0$.
The phrase ``if $A \ll_{a_1,\ldots,a_n} B$'' means ``if $A \le C_{a_1,\ldots,a_n} B$ for some $C_{a_1,\ldots,a_n} > 0$ which can be sufficiently small''.
The notation $A \sim_{ a_1, \dots, a_n} B$ means $A \lesssim_{  a_1, \dots, a_n} B$ and $B \lesssim_{  a_1, \dots, a_n} A$.

Quantities $A$ and $B$ depending on a parameter $\gamma$ are called asymptotically equal as $\gamma\to\gamma_0$ if $\lim_{\gamma\to\gamma_0} \frac{A(\gamma)}{B(\gamma)} = 1$.
This will be denoted by $A \asymp B$ (as $\gamma\to\gamma_0$).

The Fourier transform of a Schwartz function $f$ on $\mathbb{R}^d$ and the Fourier coefficients of a periodic $u \in C^\infty(\mathbb{T}^d)$ are defined respectively by
\begin{equation*}
    \widehat{f}(\xi) = (2\pi)^{-\frac{d}{2}} \int_{\mathbb{R}^d} e^{- i \xi \cdot x} f(x) \dd x,
    \qquad
    \widehat{u}_k = (2\pi)^{-\frac{d}{2}} \int_{\mathbb{T}^d} u(x) e^{- i k \cdot x} \dd x.
\end{equation*}

We will also use $\exp(x) = e^x$ to denote the exponential function.

\subsection*{Acknowledgements} 

Pierre Germain was supported by a Wolfson fellowship from the Royal Society and the Simons Collaboration Grant on Wave Turbulence. 
Hui Zhu was partly supported by the Simons Collaboration Grant on Wave Turbulence.
Iván Moyano was visiting Imperial College London between September 2023 and February 2024 supported by an ICL-CNRS fellowship. 

The authors are grateful to Nicolas Burq for his comments on an earlier version of the paper, in particular for pointing out an application to control theory.
The authors would like to thank Zeév Rudnick for his comment that simplifies the proof of Theorem~\ref{thm::QL-upperbound-3D}.
They would also like to thank Zhizhong Huang for references in analytical number theory.

Finally, the authors would like to thank the anonymous referees whose numerous valuable comments lead to the great improvement on the presentation of this work.

\section{Results and perspectives}

\subsection{Main results and organization of the article}
\label{sec::main-results}

In \S\ref{sectionpointwise}, we investigate the possible order of vanishing of an eigenfunction.
This is a simpler question than that of the behavior of local $L^2$-norms which define $m_d(\lambda,r)$, but it is closely related.
For $u \in C^\infty(\mathbb{T}^d)$, we denote by $\Gamma(u)$ its order of vanishing at the origin.
Precisely, $\Gamma(u)$ is the maximum of the set of all $N \in \mathbb{N}$ such that $\partial_x^\alpha u(0) = 0$ for all $\alpha \in \mathbb{N}^d$ with $|\alpha| \le N$.
Then we define the maximum order of vanishing
\begin{equation*}
    \Gamma_d(\lambda) = \sup_{0 \ne u \in \mathcal{E}_d(\lambda)} \Gamma(u).
\end{equation*}

The following theorem combines results proved in \S\ref{sectionpointwise} and gives upper and lower bounds for $\Gamma_d(\lambda)$. 
For simplicity, we will state the theorem only for the cases $d \ge 5$ where the estimate $\mathcal{N}_d(\lambda) \sim_d \lambda^{d-2}$ is ready to be employed.
The arguments of the proofs are rather elementary and revolve around the properties of polynomials and their zero sets.

\begin{thm}[Order of vanishing]
    For all $d \ge 5$, the following estimate holds
    \begin{equation*}
        \lambda^{\frac{d-2}{d-1}} \lesssim_d \Gamma_d(\lambda) \le 2(d-2) \lambda + \exp\Bigl(\frac{C\ln \lambda}{\ln \ln \lambda}\Bigr).
    \end{equation*}
\end{thm}

Note that the right hand side provides an explicit constant for the bound $\Gamma_d(\lambda) \lesssim_d \lambda$ which follows directly from the Donnelly--Fefferman \cite{DonnellyFefferman} (see also Logunov--Malinnikova \cite{LogunovMalinnikova}).
However, we would expect $\Gamma_d(\lambda) \sim_d \lambda^{\frac{d-2}{d-1}}$ to hold (see Conjecture~\ref{conj::vanishing-order}).

In \S\ref{sec::upperbound-vanishing}, we construct examples of eigenfunctions giving small integrals over $B_r$ and thus give upper bounds for $m_d(\lambda,r)$.
Once again, our analysis combines elementary properties of polynomials with the distribution properties of integer points on spheres.
The most general estimate for upper bounds of $m_d(\lambda,r)$ is given in Theorem~\ref{thm::upperbound-via-vanishing}.
Again, for simplicity, we will only state here the estimates when $d \ge 5$ and $r$ is sufficiently small.

\begin{thm}[Upper bounds in higher dimensions] 
\label{thm::upperbound-high-dim}
    If $d \ge 5$ and $r \ll_d \lambda^{-\frac{1}{d-1}}$, then
    \begin{equation}
    \label{eq::upperbounds-highdim-via-vanishing}
        m_d(\lambda,r) \lesssim_d \bigl( C'_d r \lambda^{\frac{1}{d-1}} \bigr)^{C_d \lambda^{\frac{d-2}{d-1}}}.
    \end{equation}
\end{thm}

In \S\ref{sec::example-degenerate}, we construct sequences of eigenfunctions of increasing eigenvalues that have bounded and small integrals over $B_r$.
These constructions provide upper bounds for $\liminf\limits_{\lambda\to\infty} m_d(\lambda,r)$.
The results we have obtained are summarized in the following theorem.

\begin{thm}[Upper bounds for limit inferiors]
    \label{thm::upper-bounds-liminf}
    For all $d \ge 2$ and $r \in (0,1)$, the following upper bounds for $\liminf\limits_{\lambda\to\infty} m_d(\lambda,r)$ hold:
    \begin{itemize}
        \item When $d=2$, we have 
        \begin{equation*}
            \liminf_{\lambda\to\infty} m_2(\lambda,r) \lesssim r^2.
        \end{equation*}
        \item When $d=3$, we have
        \begin{equation*}
            \liminf_{\lambda\to\infty} m_3(\lambda,r) \lesssim \exp\Bigl\{-\exp\Bigl(\frac{C\ln(1/r)}{\ln\ln(1/r)}\Bigr)\Bigr\}.
        \end{equation*}
        \item For all $d \ge 4$, we have
        \begin{equation*}
            \liminf_{\lambda\to\infty} m_d(\lambda,r) \lesssim \exp\bigl\{-C_dr^{3-d} \bigr\}.
        \end{equation*}
    \end{itemize}
\end{thm}

In \S\ref{sec::lowerbounds} we provide lower bounds for $m_d(\lambda,r)$. 
Tools used in the proof include distribution properties of integer points on circles and spheres (following Jarník \cite{Jarnik1926}, Connes \cite{Connes}, and Bourgain--Rudnick \cite{BourgainRudnick}) as well as some Nazarov--Turán type inequalities (see \S\ref{sec::nazarov}).

\begin{thm}[Lower bounds]
    \label{thm::lower-bounds}
    For all $d \ge 2$, the following lower bounds hold:
    \begin{itemize}
        \item When $d=2$, if $r > \lambda^{-\frac{1}{3}+ \epsilon}$, then
        \begin{equation*}
            m_2(\lambda,r) \gtrsim_\epsilon  r^2.
        \end{equation*}
        \item When $d=3$, if $r>\lambda^{-\frac{1}{12} + \epsilon}$, then
        \begin{equation*}
            m_3(\lambda,r) \gtrsim \exp\Bigl\{ -\exp\Bigl(\frac{C\ln (1/r)}{\ln\ln(1/r)}\Bigr) \Bigr\}.
        \end{equation*}
        \item For all $d \ge 4$, if $r > \lambda^{-\delta(d)}$, then for some positive constants $\delta(d)$ and $h(d)$ that depend solely on the dimension $d$, there holds
        \begin{equation*}
            m_d(\lambda,r) \gtrsim_\epsilon \exp\bigl\{-C_d r^{-h(d)-\epsilon}\bigr\}.
        \end{equation*}
    \end{itemize}
\end{thm}

Comparing with the upper bounds (for the limit inferiors) given by Theorem~\ref{thm::upper-bounds-liminf}, these lower bounds are nearly optimal when $d = 2,3$.
The statement for $d \ge 4$ is a simplified version of Theorem~\ref{thm::lower-bound-QL}, from which an explicit construction and estimates for the exponents $\delta(d)$ and $h(d)$ can be easily obtained.
Unfortunately, the exponents $h(d)$ are still far from being optimal in view of Theorem~\ref{thm::upper-bounds-liminf}.
Indeed, when $d=4$, our constructions gives $h(4) = 6$.
Moreover $h(d)$ grows factorially fast as $d\to\infty$.

\subsection{Application to quantum limits} 

The measure $\mu$ is called a quantum limit if it is the weak limit of the measures $\mu_j = |u_j|^2 \dd x$, where $u_j$ is a normalized eigenfunction associated to the eigenvalue $\lambda_j^2 \to \infty$.
Denote $\mathcal{Q}$ the set of all quantum limits.
It follows from Connes \cite{Connes} that the Lebesgue measure is absolutely continuous with respect to all quantum limits; and it was proved by Bourgain, cited in \cite{Jakobson}, that all quantum limits are absolutely continuous with respect to Lebesgue measure.
See \cite{Jakobson, Aissiou} for further results on quantum limits.
We will denote
\begin{equation*}
    m_d(r) = \inf_{\mu \in \mathcal{Q}} \frac{\mu(B_r)}{|B_r|},
\end{equation*}
so that $m_d(r)$ quantifies the infimum of the measure that a quantum limit can have on a ball of radius $r$.
In \S\ref{sec::QL-relation}, we prove the the identity
\begin{equation}
    \label{eq::QL-relation}
    m_d(r) = \liminf_{\lambda\to\infty} m_{d}(\lambda,r).
\end{equation}

An immediate consequence of Theorems \ref{thm::upper-bounds-liminf} and \ref{thm::lower-bounds} is the following corollary.

\begin{cor}
    \label{cor::quantum-limit}
    For all $d\ge 2$ and $r \in (0,1)$, the following bounds for $m_d(r)$ hold:
    \begin{itemize}
        \item When $d=2$, we have
        \begin{equation*}
            m_2(r) \sim  r^2.
        \end{equation*}
        \item When $d=3$, we have
        \begin{equation*}
            \exp\Bigl\{ -\exp\Bigl(\frac{C\ln (1/r)}{\ln\ln(1/r)}\Bigr) \Bigr\} \lesssim m_3(r) \lesssim \exp\Bigl\{ -\exp\Bigl( \frac{C'\ln (1/r)}{\ln\ln(1/r)} \Bigr) \Bigr\}.
        \end{equation*}
        \item When $d\ge 4$, we have
        \begin{equation*}
            \exp\bigl\{-C_d r^{-h(d)-\epsilon}\bigr\} \lesssim_\epsilon m_d(r) \lesssim \exp\bigl\{-C'_d r^{3-d} \bigr\}.
        \end{equation*}
    \end{itemize}
\end{cor}

It is interesting to compare the behavior of quantum limits on tori with their behavior on the sphere or on hyperbolic surfaces.
Note that these are respectively manifolds of zero, positive, and negative curvature.
Roughly speaking, the corollary above expresses the fact that, near any point of a torus, quantum limits can vanish at most polynomially or exponentially.
This is intermediary between the case of the sphere, where quantum limits of highest weight spherical harmonics may vanish completely on some nonempty open sets, and the case of hyperbolic surfaces, where quantum unique ergodicity is expected (and indeed has been proved by Lindenstrauss \cite{Lindenstrauss} in the arithmetic case).

\subsection{Application to control theory} 

The controllability of the linear Schrödinger group $e^{it\Delta}$ without potential on tori has been established by Jaffard \cite{Jaffard} when $d = 2$ and by Komornik \cite{Komornik} and Macià \cite{Macia2011} using different methods for all $d\ge 2$.

Via a classical orthogonality trick, the observability of toral eigenfunctions immediately implies the observability --- and thus the controllability (by the Hilbert uniqueness method) --- of $e^{it\Delta}$ within the time interval $[0,2\pi]$:
\begin{align*}
    \int_0^{2\pi} \|e^{it\Delta} u_0\|_{L^2(B_r)}^2 \dd t
    & = \int_{B_r} \int_0^{2\pi} \biggl| \sum_{n\ge 0} e^{itn} \Pi_n u_0(x) \biggr|^2 \dd t \dd x
    = \sum_{n\ge 0} \int_{B_r} |\Pi_n u_0(x)|^2 \dd x \\
    & \ge \inf_{n \in \mathbb{N}} m_d(\sqrt{n},r) \sum_{n\ge 0} \int_{\mathbb{T}^d} |\Pi_n u_0(x)|^2 \dd x
    = \inf_{n \in \mathbb{N}} m_d(\sqrt{n},r) \|u_0\|_{L^2}^2,
\end{align*}
where $\Pi_n$ denotes the Fourier projector onto $\mathcal{E}_d(\sqrt{n})$.

For better observability estimates, e.g., with better constants and within shorter time periods, recall that the general functional analytic framework by Miller \cite{Miller} implied that the controllability of the Schrödinger group from $B_r$ is equivalent to the resolvent estimate
\begin{equation}
    \label{eq::inequality-control}
    \| u \|_{L^2}^2 \le m {\|u\|^2_{L^2(B_r)}} + M \| (\Delta + \tau^2) u \|_{L^2}^2,
\end{equation}
for some finite $m > 0$, $M > 0$, and for all $u \in C^\infty(\mathbb{T}^d)$ and $\tau \in \mathbb{R}$.
Furthermore, this estimate implies the observability
\begin{equation*}
     \|u_0\|_{L^2}^2 \le \frac{2mT}{T^2-\pi^2M} \int_0^T \|e^{it\Delta} u_0\|_{L^2(B_r)}^2 \dd t \qquad \forall T > \pi \sqrt{M}.
\end{equation*}

An inequality similar to \eqref{eq::inequality-control}  can be obtained immediately from the definition of $m_d(\lambda,r)$. Indeed, choose $\lambda^2 \in \mathbb{R}$, let $n$ be the closest integer, and split $u = u_0 + u_1$ where $u_0 \in \mathcal{E}_{n}$ and $u_1$ is orthogonal to $\mathcal{E}_{n}$.
As a consequence of our estimates, we get that
\begin{align*}
    \| u \|_{L^2}
    & \le \| u_0 \|_{L^2} + \| u_1 \|_{L^2} \\
    & \lesssim m_d(\sqrt n,r)^{-1} r^{-d} \| u_0 \|_{L^2(B_r)} + \| (\Delta + \lambda^2) u \|_{L^2} \\
    & \le m_d(\sqrt n,r)^{-1} r^{-d} \| u \|_{L^2(B_r)} + m_d(\sqrt n,r)^{-1} r^{-d} \| u_1 \|_{L^2(B_r)} + \| (\Delta + \lambda^2) u \|_{L^2} \\
    & \le m_d(\sqrt n,r)^{-1} r^{-d} \bigl[ \| u \|_{L^2(B_r)} +  \| (\Delta + \lambda^2) u \|_{L^2} \bigr].
\end{align*}
While this inequality is of the form \eqref{eq::inequality-control}, it is most likely not optimal.
We believe that it could be substantially improved by applying the methods laid out in the present paper to functions supported on spectral bands instead of eigenfunctions. Indeed, denoting by $m_d(\lambda,\delta,r)$ the optimal observability constant for $u$ supported spectrally (for $\sqrt{-\Delta}$) on $[\lambda-\delta,\lambda+\delta]$, the same proof as above gives
\begin{equation*}
    \| u \|_{L^2} \lesssim m_d(\sqrt n,\delta,r)^{-1} r^{-d} \bigl[ \| u \|_{L^2(B_r)} + (\lambda \delta)^{-1} \| (\Delta + \lambda^2) u \|_{L^2} \bigr].
\end{equation*}
We do not attempt to go any further in this direction, but it seems potentially interesting.

\subsection{Open questions} 

The results which have been stated so far give a very lacunary picture, and much remains to be understood.
We start with two conjectures, for which strong evidence is available.
First, we conjecture that the examples constructed in \S\ref{sectionpointwise} give the highest possible order of vanishing.
See Remark~\ref{rem::vanishing-order-discuss} for an heuristic discussion.

\begin{conj}
    \label{conj::vanishing-order}
    For all $d \geq 2$, there holds
    \begin{equation*}
        \Gamma_d(\lambda) \sim_d \mathcal{N}_d(\lambda)^{\frac{1}{d-1}}.
    \end{equation*}
\end{conj}

Second, we conjecture that the optimal lower bound for $m_d(r)$ when $d \ge 4$ matches the upper bound given in Corollary \ref{cor::quantum-limit}.

\begin{conj}
    If $d\ge 4$, then
    \begin{equation*}
        \exp\bigl\{-C_d r^{3-d}\bigr\} \lesssim m_d(r) \lesssim \exp\bigl\{-C'_d r^{3-d} \bigr\}.
    \end{equation*}
\end{conj}

Beyond these two conjectures, it is tempting to ask if the upper bounds for $m_d(\lambda,r)$ given in \S\ref{sec::example-upperbound} are are optimal and if matching lower bounds exist, at least when $d \ge 5$.
However, our construction of the examples might not be sharp. 
The distribution of integer points in low dimensions ($d=2,3,4$) is more irregular, and the upper bounds we are able to prove might be far off the mark.

\section{Other geometries}
\label{sec::other-geometries}

Consider the analog on a general compact Riemannian manifold $M$ by setting
\begin{equation}
    \label{eq::general-m}
    m_{M,x_0}(\lambda,r) = \inf_{0\ne u \in \mathcal{E}_M(\lambda)} \frac{\displaystyle \fint_{B_r(x_0)} |u(x)|^2 \dd x}{\displaystyle \fint_{M} |u(x)|^2 \dd x},
\end{equation}
where $\mathcal{E}_M(\lambda)$ is the set of all eigenfunctions of the Laplace-Beltrami operator $-\Delta_M$ with eigenvalue $\lambda^2$ and $B_r(x_0)$ is the geodesic ball centered at $x_0 \in M$ with radius $r > 0$.
Clearly, we need $r < \operatorname{Inj}_{x_0} M$, the injective radius at $x_0$.

\subsection{General manifolds}

Considering a general compact manifold $M$, following the doubling inequality of Donnelly--Fefferman \cite{DonnellyFefferman}, Logunov and Malinnikova \cite{LogunovMalinnikovaICM, LogunovMalinnikova} proved the Remez inequality
\begin{equation*}
    \| u \|_{L^\infty(M)} \lesssim \| u \|_{L^\infty(E)} \Bigl( \frac{C |M|}{|E|} \Bigr)^{C \lambda}
\end{equation*}
for all $u \in \mathcal{E}_M(\lambda)$, all Borel $E \subset M$ with positive measure.
As in the proof of Lemma~\ref{lem::nazarov}, it can be deduced from this Remez inequality that (for some different constant $C$):
\begin{equation}
    \label{remezgeneral}
    m_{M,x_0}(\lambda,r) \gtrsim (Cr)^{C\lambda}.
\end{equation}

\subsection{Spheres}

On the sphere $\mathbb{S}^d$, we will investigate the quotient
\begin{equation*}
    \frac{\displaystyle \fint_{B_r(x_0)} |u(x)|^2 \dd x}{\displaystyle \fint_{\mathbb{S}^d} |u(x)|^2 \dd x}
\end{equation*}
on the highest weight spherical harmonics $H_n(x) = n^{\frac{d-1}{4}} z^n$ where $x = (x^1,\ldots,x^{d+1}) \in \mathbb{R}^{d+1}$, $z=x^1+ix^2 \in \mathbb{C}$ and $n \in \mathbb{N}$.
These are eigenfunctions of the spherical Laplacian associated with eigenvalue $\lambda_n^2 = n(d-1+n)$.
They are well-known to maximize various inequalities related to the concentration of eigenfunctions.
It is classical that $\|H_n\|_{L^2(\mathbb{S}^d)} \sim 1$ (see e.g., \cite[\S4]{Sogge2016Lp}).
For $x_0 = (0,0,1,0,\dots,0)$ and $r \ll 1$, we obtain
\begin{equation*}
    m_{\mathbb{S}^d,x_0}(\lambda_n,r) 
    \le \fint_{B_r(x_0)} |H_n(x)|^2 \dd x
    \sim_d n^{\frac{d-1}{2}} \fint_{|z|\le C_d r} |z|^{2n} |\dd z|
    \sim n^{\frac{d-1}{2}} (C_d r)^{2n},
\end{equation*}
which gives the sharpness of the lower bound \eqref{remezgeneral}.

\subsection{Surfaces of negative curvature} 

When $M$ is a compact surface of negative curvature, Hezari--Rivière \cite{HezariRiviere} showed that if $r \gtrsim (\ln \lambda)^{-1}$, then
\begin{equation*}
    m_{M,x_0}(\lambda,r) \sim_M 1.
\end{equation*}
In the case of arithmetic surfaces, much stronger results can be obtained, see Luo--Sarnak \cite{LuoSarnak} and Young \cite{Young}.

\subsection{Irrational tori and spectral bands} 

We consider now a general torus $\mathbb{T}^d_\Lambda = \mathbb{R}^d / \Lambda$, where $\Lambda$ is a lattice of full dimension. 
The definition \eqref{eq::general-m} should be modified in this case since, for a generic choice of $\Lambda$, the eigenspaces have bounded dimension (equal to $2$).
Instead, projection on eigenspaces should be replaced by projection on narrow spectral bands.
Precisely, we propose the definition
\begin{equation}
    \label{neweq::def-m}
    m_\Lambda(\lambda, r) = \inf_{u \in \mathcal{E}_\Lambda\bigl(\lambda - \frac{1}{2\lambda}, \lambda + \frac{1}{2\lambda}\bigr) } \frac{\displaystyle \fint_{B_r} |u(x)|^2 \dd x}{\displaystyle \fint_{\mathbb{T}^d_\Lambda} |u(x)|^2 \dd x},
\end{equation}
where $\mathcal{E}_\Lambda\bigl(\lambda - \frac{1}{2\lambda}, \lambda + \frac{1}{2\lambda}\bigr)$ is by definition the union of all eigenspaces whose eigenvalues lie between $\lambda - \frac{1}{2\lambda}$ and $\lambda + \frac{1}{2\lambda}$. 
The choice of the bandwidth $\frac{1}{2\lambda}$ ensures that the definition \eqref{neweq::def-m} and the original definition \eqref{eq::def-m} essentially agree. Indeed, eigenvalues of the Laplacian on $\mathbb{T}^d$ are given by $\mathbb{N}$.
The spacing between two consecutive eigenvalues is therefore equal to $1$ --- at least when $d \ge 4$ --- hence the spacing between two consecutive $\lambda$ for which $\lambda^2$ is an eigenvalue is $\sim \frac{1}{\lambda}$.

With the definition \eqref{neweq::def-m} above, do the theorems proved in the present paper extend to general $\Lambda$? Only partly. Roughly speaking, the arguments relying on Jarnik's lemma (Lemma~\ref{lem::jarnik}) or more generally Connes' result (Lemmas~\ref{lem::connes} and~\ref{lem::jarnik-connes}) still apply.
However, many counting estimates (e.g., number of lattice points on spheres in low dimension) for general $\Lambda$ are not yet well understood.

Finally, it would certainly be of interest to consider spectral bands of arbitrary width, on tori but also on other compact manifolds, but this goes well beyond the scope of the present paper.

\section{Pointwise vanishing}
\label{sectionpointwise}

In this section, we obtain bounds for $\Gamma_d(\lambda)$ (recalling its definition in \S\ref{sec::main-results}) when $\lambda \gg 1$.
\begin{itemize}
    \item In \S\ref{construction}, we give the lower bound
    \begin{equation*}
        \Gamma_d(\lambda) \gtrsim (d-1) \mathcal{N}_d(\lambda)^{\frac{1}{d-1}}.
    \end{equation*}
    \item In \S\ref{sec::maximal-Gamma}, we give two upper bounds: 
    \begin{equation*}
        \Gamma_d(\lambda) \le \mathcal{N}_d(\lambda)-2,
        \quad
        \Gamma_d(\lambda) \le 2(d-2)\lambda + \lambda^\epsilon,
    \end{equation*}
    for all $\epsilon > 0$, where a negative value on the right hand side means the nonexistence of eigenfunctions which vanishes at the origin.
    Note that when $d \ge 5$, the second bound gives a sharper estimate; whereas when $d=2$, the first bound is better.
\end{itemize}

\subsection{Eigenfunctions with high vanishing orders} 
\label{construction}

The following theorem gives an improvement of the result in Täufer \cite{Taufer}.
It is quantitative as it gives a lower bound for $\Gamma_d(\lambda)$, and valid in any dimension. We believe that this example is essentially optimal.
However we are far from being able to match this lower bound with an upper bound for $\Gamma_d(\lambda)$ (see \S\ref{sec::maximal-Gamma} for partial results).

\begin{thm}
    \label{thm::vanishing}
    For all $d \ge 2$, we have $\Gamma_d(\lambda) \ge N \in \mathbb{N}$ if
    \begin{equation*}
        \binom{N+d-1}{d-1} + \binom{N+d-2}{d-1} < \mathcal{N}_d(\lambda).
    \end{equation*}
    Consequently, there holds
    \begin{equation*}
        \Gamma_d(\lambda) \gtrsim (d-1)\mathcal{N}_d(\lambda)^{\frac{1}{d-1}}.
    \end{equation*}
    Particularly, if $d \ge 5$ then
    \begin{equation*}
        \Gamma_d(\lambda) \gtrsim (d-1) \lambda^{\frac{d-2}{d-1}}.
    \end{equation*}
\end{thm}
\begin{proof}
    Expanding the Fourier series of $u$ in Taylor series at $0$, we find that $\Gamma(u)\ge N$ if and only if there holds
    \begin{equation}
        \label{eq::hirondelle}
        \sum_{k \in \mathcal{S}_d(\lambda)} \widehat{u}_k k^\alpha =0, \qquad \forall \alpha \in \mathbb{N}^d, \ |\alpha|\le N,
    \end{equation}
    where $k^\alpha = \prod_{j=1}^d k_j^{\alpha_j}$ and $|\alpha| = \sum_{j=1}^d \alpha_j$.
    Using $|k|^2 = \sum_{j=1}^d k_j^2 = \lambda^2$ for $k \in \mathcal{S}_d(\lambda)$, it suffices for \eqref{eq::hirondelle} to hold for all $\alpha \in \Sigma_0 \cup \Sigma_1$, where $\Sigma_j$ is the set of all $\alpha \in \mathbb{N}^d$ with $|\alpha| \le N$ and $\alpha_d=j$.
    Consider the linear map $\Phi:\mathbb{C}^{\mathcal{S}_d(\lambda)} \to \mathbb{C}^{ \Sigma_0 \cup \Sigma_1}$ defined by
    \begin{equation*}
        \Phi : (c_k)_{k \in \mathcal{S}_d(\lambda)} \mapsto \biggl( \sum_{k \in \mathcal{S}_d(\lambda)} c_k k^\alpha \biggr)_{\alpha \in \Sigma_0 \cup \Sigma_1}.
    \end{equation*}
    Clearly if $u \in \ker \Phi$, then $\Gamma(u) \ge N$.
    By the rank-nullity theorem, $\ker\Phi\ne\{0\}$ when $\mathcal{N}_d(\lambda) > \# \Sigma_0 + \# \Sigma_1$.
    Note that (see e.g., \cite[Lem~2.2]{Guth}):
    \begin{equation*}
        \# \Sigma_0 = \binom{N+d-1}{d-1},
        \quad
        \#\Sigma_1 = \binom{N+d-2}{d-1},
    \end{equation*}
    which implies, as $N \to \infty$,
    \begin{equation*}
        \# \Sigma_0 + \# \Sigma_1 \asymp \frac{2 N^{d-1}}{(d-1)!}.
    \end{equation*}
    Therefore $\Gamma_d(\lambda) \ge N$ provided
    \begin{equation*}
        N \ll [(d-1)!]^{\frac{1}{d-1}} \mathcal{N}_d(\lambda)^{\frac{1}{d-1}}.
    \end{equation*}
    We conclude with $[(d-1)!]^{\frac{1}{d-1}} \gtrsim d-1$.
\end{proof}

The exact same proof extends to general complex exponential polynomials.
In fact, if for any subset $\Lambda \subset \mathbb{R}^d$ let $\Gamma_d(\Lambda)$ be the supremum of $\Gamma(u)$ among all nonzero exponential polynomials of the form $u(x) = \sum_{k \in \Lambda} c_k e^{i k \cdot x}$, then the following theorem holds.

\begin{thm}
    \label{thm::vanishing-general}
    Let $d \ge 2$ and let $\Lambda \subset \lambda \mathbb{S}^{d-1}$ be a nonempty finite set.
    Then $\Gamma_d(\Lambda) \ge N$ if
    \begin{equation*}
        \binom{N+d-1}{d-1} + \binom{N+d-2}{d-1} < \# \Lambda.
    \end{equation*}
    Consequently, there holds
    \begin{equation*}
        \Gamma_d(\Lambda) \gtrsim (d-1) (\# \Lambda)^{\frac{1}{d-1}}.
    \end{equation*}
\end{thm}

\subsection{Maximal order of vanishing}
\label{sec::maximal-Gamma}

It follows from Donnelly--Fefferman \cite{DonnellyFefferman} that 
\begin{equation*}
    \Gamma_d(\lambda) \lesssim_d \lambda.
\end{equation*}
See also Logunov--Malinnikova \cite[Prop~2.4.1]{LogunovMalinnikova} for a simpler approach.
This result holds on any compact manifold, we aim to improve it on tori.
As already noted in \cite{BourgainRudnick2011bis}, Nazarov's theorem \cite[Theorem I]{Nazarov} (we give an extension of this theorem to higher dimensions in Lemma~\ref{lem::nazarov}) can be used to give a lower bound for $\Gamma_d(\lambda)$.
In the following lemma, we use this idea with more precision.

\begin{lem}
    For all $d \ge 2$ and $u \in \mathcal{E}_d(\lambda)$, there holds
    \begin{equation}
        \label{eq::lower-bound-infty-nazarov+FM}
        \| u \|_{L^\infty(B_r)} \ge (Cr)^{\mathcal{M}_d(\lambda)} \| u \|_{L^\infty}, \quad \mathcal{M}_d(\lambda) = \min \{ \mathcal{N}_d(\lambda) - 1, 2d \lambda \}.
    \end{equation}
    Consequently, we have
    \begin{equation*}
        \Gamma_d(\lambda) \le \mathcal{M}_d(\lambda)-1.
    \end{equation*}
\end{lem}

The estimate \eqref{eq::lower-bound-infty-nazarov+FM} follows directly from \eqref{eq::nazarov-infty} and \eqref{eq::FM-finer-ineq}.
We now give a proof of
\begin{equation*}
    \|u\|_{L^\infty(B_r)} \ge (Cr)^{2d\lambda} \|u\|_{L^\infty},
\end{equation*}
which parallels \eqref{eq::FM-finer-ineq} but in a more illustrative form.
    
\begin{proof}
    Let $x_0 \in B_{r/2}$ and let $\bm{e}_j$ be the unit vector in direction of the $j$-th axis.
    Consider the exponential polynomial $f\in C^\infty(\mathbb{T}^1)$ defined by
    \begin{equation*}
        f(t) = u(x_0 + t \bm{e}_j).
    \end{equation*}
    The number of distinct exponents in $f$ is at most $2\lambda+1$ for they are projections of $\mathcal{S}_d(\lambda) \subset \lambda \mathbb{S}^{d-1}$ to $\mathbb{R}\bm{e}_j$ and lie in an interval of length $\le 2\lambda$.
    By Nazarov's inequality \eqref{eq::nazarov-infty} with $d=1$, we have
    \begin{equation*}
        \| f \|_{L^\infty(-r,r)} \ge (Cr)^{2\lambda} \| f \|_{L^\infty(\mathbb{T})}.
    \end{equation*}
    This implies
    \begin{equation*}
        \| u \|_{L^\infty(E_1)} \le (Cr)^{-2\lambda} \| u \|_{L^\infty(B_r)},
    \end{equation*}
    where $E_1$ is the union of all geodesics on $\mathbb{T}^d$ in the $\bm{e}_1$ direction which intersect with $B_{r/2}$. 
    
    One then iterates this construction: let $E_2$ be the union of all geodesics on $\mathbb{T}^d$ in the $\bm{e}_2$ direction which intersect with $E_{1}$. At each point of $E_2$, we apply the same argument (considering now $f(t) = u(x_0 +te_2)$) to obtain
    \begin{equation*}
        \|u\|_{L^\infty(E_2)} \le (Cr)^{-2\lambda} \|u\|_{L^\infty(E_{1})} \le (Cr)^{-4\lambda} \|u\|_{L^\infty(B_r)}.
    \end{equation*}
    At the $j$-th iterate, one obtains
    \begin{equation*}
        \|u\|_{L^\infty(E_j)} \le (Cr)^{-2\lambda} \|u\|_{L^\infty(E_{j-1})} \le (Cr)^{-2j\lambda} \|u\|_{L^\infty(B_r)},
    \end{equation*}
    where $E_j$ is the union of all geodesics on $\mathbb{T}^d$ in the $\bm{e}_j$ direction which intersect with $E_{j-1}$.
    We finish the proof at the $d$-th iterate.

    To see the upper bound for $\Gamma_d(\lambda)$, note that if $\Gamma(u) = N$ for some nonzero $u \in \mathcal{E}_d(\lambda)$ then $\|u\|_{L^\infty(B_r)} \lesssim_u r^{N+1}$.
    Also $\|u\|_{L^\infty} \gtrsim_u 1$ since $u \not\equiv 0$.
    Therefore by \eqref{eq::lower-bound-infty-nazarov+FM}, we have $r^{N+1} \gtrsim_u (Cr)^{\mathcal{M}_d(\lambda)}$.
    Choosing $0 < r \ll_{d,\lambda,u} 1$ yields $N+1 \le \mathcal{M}_d(\lambda)$.
\end{proof}

Note that when $\lambda \gg 1$, $\mathcal{M}_d(\lambda) = \mathcal{N}_d(\lambda) -1$ for $d=2$, $\mathcal{M}_d(\lambda) = 2 d \lambda$ for $d \ge 5$.
For $d=3,4$, we may have either $\mathcal{M}_d(\lambda) = \mathcal{N}_d(\lambda) -1$ or $\mathcal{M}_d(\lambda) = 2 d \lambda$, depending on the actual values of $\lambda$.
As a simple example, we have $\mathcal{M}_4(2^n) \le \mathcal{N}_4(2^n) -1 \lesssim 1$.
The estimate \eqref{eq::lower-bound-infty-nazarov+FM} applies to all functions with Fourier coefficients supported in balls of radius $\lambda$, which are much larger sets than $\mathcal{S}_d(\lambda)$.
We now turn to a more elementary method of proof, which relies solely on properties of polynomials and gives a slight improvement of the preceding bound.

\begin{thm} 
\label{thm::vanishing-order-upperbound}
    For all $d \ge 2$, there holds
    \begin{equation*}
        \Gamma_d(\lambda) \le D_d(\lambda)-1, \qquad
        D_d(\lambda) = \begin{cases}
            \mathcal{N}_d(\lambda)-1, & d=2;\\
            2(d-2) \lambda + \exp\bigl\{\frac{C\ln \lambda}{\ln \ln \lambda}\bigr\}, & d \ge 3.
        \end{cases}
    \end{equation*}
\end{thm}

\begin{proof}
    We prove by induction on the dimension $d$ that: for all $\lambda$ and all $\ell \in \mathcal{S}_d(\lambda)$, there exists a polynomial $P_\ell \in \mathbb{R}[X_1,\dots,X_d]$ with $\deg P_\ell \le D_d(\lambda)$ such that $P_\ell(k) = \bm{1}_{k=\ell}$ for all $\ell \in \mathcal{S}_d(\lambda)$.
    Linear combinations of \eqref{eq::hirondelle} yield that 
    \begin{equation*}
        \sum_{k \in \mathcal{S}_d(\lambda)} \widehat{u}_k P(k) = 0
    \end{equation*}
    for all $P \in \mathbb{P}[X_1,\ldots,X_d]$ with $\deg P \le \Gamma_d(\lambda)$ and for all $u \in \mathcal{E}_d(\lambda)$.
    Therefore, if we assume by contradiction that $\Gamma_d(\lambda) \ge D_d(\lambda)$, then
    \begin{equation*}
        \widehat{u}_\ell = \sum_{k \in \mathcal{S}_d(\lambda)} \widehat{u}_k P_\ell(k) = 0
    \end{equation*}
    for all $\ell \in \mathcal{S}_d(\lambda)$ which implies $u=0$.

    To construct such polynomials, when $d=2$ we simply let
    \begin{equation*}
        P_\ell(X_1,X_2) = \Re \prod_{k \in \mathcal{S}_2(\lambda) \backslash \{\ell\}} \frac{(X_1+iX_2)-(k_1+ik_2)}{(\ell_1+i\ell_2)-(k_1+ik_2)},
    \end{equation*}
    where $\Re z$ denotes the real part of a complex number $z$.
    When $d \ge 3$, write $\ell = (\ell',\ell_d)$ where $\ell' \in \mathbb{Z}^{d-1}$ and $\ell_d \in \mathbb{Z}$.
    Let $\lambda' = |\ell'|$ and put
    \begin{equation*}
        P_\ell(X_1,\ldots,X_d)
        = Q_{\ell'}(X_1,\ldots,X_{d-1}) \prod_{n \in [-\lambda,\lambda] \cap \mathbb{Z} \backslash \{\ell_d\}} \frac{X_d - n}{\ell_d - n},
    \end{equation*}
    where $Q_{\ell'} \in \mathbb{R}[X_1,\ldots,X_{d-1}]$ satisfies $\deg Q_{\ell'} \le D_{d-1}(\lambda')$ and $Q_{\ell'}(k') = \bm{1}_{k'=\ell'}$ for all $k' \in \mathcal{S}_{d-1}(\lambda')$.
    Clearly, if $d=2$ then $\deg P_\ell = \mathcal{N}_2(\lambda)-1$; if $d \ge 3$, then by the induction hypothesis, there holds
    \begin{equation*}
        \deg P_\ell \le \deg Q_{\ell'} + 2\lambda
        \le D_{d-1}(\lambda') + 2\lambda \le \exp\Bigl\{\frac{C\ln \lambda'}{\ln \ln \lambda'} \Bigr\} + 2(d-3)\lambda' + 2\lambda.\qedhere
    \end{equation*}
\end{proof}

\begin{rem}
    \label{rem::vanishing-order-discuss}
    It seems likely that the optimal estimate is
    \begin{equation*}
        \Gamma_d(\lambda) \lesssim \mathcal{N}_d(\lambda)^{\frac{1}{d-1}}.
    \end{equation*}
    This is a consequence of the following heuristic argument.
    As in the proof above, it suffices to find, for all $\lambda$ and $\ell \in \mathcal{S}_d(\lambda)$, a polynomial $P_\ell \in \mathbb{R}[X^1,\ldots,X^d]$ with $\deg P_\ell \lesssim \mathcal{N}_d(\lambda)^{\frac{1}{d-1}}$ and which satisfies $P_\ell(k) = \bm{1}_{k=\ell}$ for all $k \in \mathcal{S}_d(\lambda)$.

    Choosing a projection onto a hyperplane $\Pi : \mathbb{R}^d \to V \simeq \mathbb{R}^{d-1}$ such that $\Pi|_{\mathcal{S}_d(\lambda)}$ is injective, it suffices to find, for all $\ell \in \mathcal{S}_d(\lambda)$, a polynomial $Q_\ell \in \mathbb{R}[X_1,\ldots,X_{d-1}]$ with $\deg Q_\ell \lesssim \mathcal{N}_d(\lambda)^{\frac{1}{d-1}}$ and satisfies $Q_\ell(\Pi k) = \bm{1}_{k=\ell}$ for all $k \in \mathcal{S}_d(\lambda)$.

    Let $\mathcal{P}_N = \{P \in \mathbb{R}[X_1,\ldots,X_{d-1}] \mid \deg P \le N\}$.
    Then
    \begin{equation*}
        \dim \mathcal{P}_N = \binom{N+d-1}{d-1} \asymp \frac{N^{d-1}}{(d-1)!}.
    \end{equation*}
    For $\ell \in \mathcal{S}_d(\lambda)$, by the rank-nullity theorem, the linear map
    \begin{equation*}
        \Phi_\ell : \mathcal{P}_N \to \mathbb{R}^{\mathcal{N}_d(\lambda)-1},
        \quad
        P \mapsto (P(\Pi k))_{k \in \mathcal{S}_d(\lambda) \backslash \{k\}}
    \end{equation*}
    has a nontrivial kernel if $\dim \mathcal{P}_N > \mathcal{N}_d(\lambda)-1$, and consequently when
    \begin{equation*}
        N \gg_d \mathcal{N}_d(\lambda)^{\frac{1}{d-1}}.
    \end{equation*}
    Let us choose a nonzero $P_\ell \in \ker \Phi_\ell$ with minimal degree.
    We would have found the desired polynomial if $\ell \notin Z_\ell$, the zero set of $P_\ell$.
    The chance for this to happen is definitely not small for $Z_\ell$ has zero measure.
    In fact, this argument applies unless elements of $\mathcal{S}_d(\lambda)$ satisfy some non-trivial algebraic relations.
\end{rem}

\section{Examples and upper bounds}
\label{sec::example-upperbound}

The aim of this section is to construct examples of eigenfunctions with eigenvalue $\lambda^2$ achieving small $L^2$-norms on balls of radius $r$.
This is of course equivalent to bounding $m_d(\lambda,r)$ from above.

Our main estimate is stated in Theorem~\ref{thm::upperbound-via-vanishing}, which considers eigenfunctions whose Fourier coefficients are supported in an arbitrary $\Lambda \subset \mathcal{S}_d(\lambda)$.
This estimate depends on the number of lattice points in $\Lambda$ and the diameter of $\Lambda$.
Corollary~\ref{cor::upperbound-high-dimension} considers the special cases where $\Lambda$ is the set of all integer points in a spherical cap, giving upper bounds that rely on the competition between the small radius $r$ and the density of integer points on spherical caps.
Finally in \S\ref{sec::example-degenerate}, we construct examples with $\Lambda = \mathcal{S}_d(\lambda) \cap H$ where $H$ is a hyperplane. For these examples we obtain estimates independent of $\lambda$ and thus give upper bounds for $m_d(r)$.

\subsection{A simple example}

First let us give a simple example which gives the upper bound $m_d(r) \lesssim r^2$ in all dimensions.
In \S\ref{sec::lowerbounds}, we will see that $m_2(r) \gtrsim r^2$, and therefore $m_2(r) \sim r^2$.
In \S\ref{sec::example-degenerate} we will give an improvement of this upper bound in higher dimensions.

\begin{lem}
    For all $d \ge 2$ and $r\in(0,1)$, we have $m_d(r)  \lesssim r^2$.
\end{lem}
\begin{proof}
    For $n \in \mathbb{N}$, let $a_n=(n,n+1,0,\ldots,0)$ and $b_n=(n+1,n,\ldots,0) \in \mathcal{S}_d(\lambda_n)$ where
    \begin{equation*}
        \lambda_n = \sqrt{n^2+(n+1)^2},
    \end{equation*}
    which tends to infinity as $n\to\infty$.
    Then $|a_n-b_n| = \sqrt{2}$.
    Writing $x = (x_1,\ldots,x_d)$, let
    \begin{equation*}
        u_n(x)
        = e^{i a_n\cdot x} - e^{i b_n \cdot x}
        = e^{i n(x_1+x_2)} (e^{i x_1}-e^{i x_2}).
    \end{equation*}
    Then $u_n \in \mathcal{E}_d(\lambda_n)$ with $\| u \|_{L^2} \sim 1$.
    We conclude with \eqref{eq::QL-relation} and the following estimate:
    \begin{align*}
        m_d(\lambda_n,r)
        \le \fint_{B_r}|u_n(x)|^2  \dd x
        & = \fint_{B_r} |e^{i x_1}-e^{i x_2}|^2  \dd x \\
        & = \fint_{B_r} \bigl(|x_1-x_2| + \mathcal{O}(r^2)\bigr)^2  \dd x
        \sim r^2.
        \qedhere
    \end{align*}
\end{proof}

\subsection{Eigenfunctions with high order pointwise vanishing}
\label{sec::upperbound-vanishing}

Using results on the pointwise vanishing of eigenfunctions, we can give upper bounds on $m_d(\lambda,r)$.

\begin{thm}
    \label{thm::upperbound-via-vanishing}
    For all $d \ge 2$ and all nonempty $\Lambda \subset \mathcal{S}_d(\lambda)$, there holds:
    \begin{equation}
    \label{eq::upperbound-via-vanishing}
        m_d(\lambda,r)
        \lesssim_d \frac{(\# \Lambda)^{\frac{1}{2}}}{\sqrt{\Gamma_d(\Lambda)+1}} \Bigl(\frac{e\cdot r\diam \Lambda}{\Gamma_d(\Lambda)+1}\Bigr)^{\Gamma_d(\Lambda)+1},
    \end{equation}
    where $\Gamma_d(\Lambda)$ was defined above Theorem~\ref{thm::vanishing-general}.
\end{thm}
\begin{proof}
    Choose an eigenfunction $u \in \mathcal{E}_d(\lambda)$ whose Fourier coefficients are supported in $\Lambda$ and satisfies $\Gamma(u) = \Gamma_d(\Lambda) = N$. Let $k_0 \in \Lambda$ and set $v(x) = e^{-i k_0 \cdot x} u(x)$, so that $v$ still vanishes to order $N$ at the origin, but its Fourier coefficients are supported in $\Lambda-k_0$.
    Taylor’s formula with remainder applied to the function $t \mapsto v(tz)$ (with $z \in \mathbb{R}^d$) gives
    \begin{equation*}
        v(z) = v(0) + z \cdot \nabla v(0) + \dots + \frac{1}{N!} (z \cdot \nabla)^N v(0)+ \frac{1}{(N+1)!} (z \cdot \nabla)^{N+1} v(sz),
    \end{equation*}
    with $0 \leq s \leq 1$. By choice of $v$, all summands but the last one vanish on the right-hand side. Furthermore, the Fourier series of $(z \cdot \nabla)^{N+1} v(x)$ have coefficients $(z\cdot k)^{N+1} \widehat{v}_k$. This leads to the bound
    \begin{equation*}
        \| v \|_{L^\infty(B_r)} \lesssim \frac{1}{(N+1)!} \sup_{z \in B_r} \sum_{k \in \Lambda - k_0} |z\cdot k|^{N+1} | \widehat{v}_k | \leq  \frac{(r\diam \Lambda)^{N+1}}{(N+1)!} \|\widehat{v}\|_{\ell^1(\Lambda-k_0)}.
    \end{equation*}
    
    By Cauchy's inequality, there holds
    \begin{equation*}
        \|\widehat{v}\|_{\ell^1(\Lambda-k_0)}^2
        = \|\widehat{u}\|_{\ell^1(\Lambda)}^2
        \le \# \Lambda \|\widehat{u}\|_{\ell^2(\Lambda)}
        = \# \Lambda \|\widehat{u}\|_{\ell^2}.
    \end{equation*}
    Therefore,
    \begin{equation*}
        \|u\|_{L^\infty(B_r)}
        \le \frac{(r\diam \Lambda)^{N+1}}{(N+1)!} \|\widehat{u}\|_{\ell^1(\Lambda)}
        \le \frac{(r\diam \Lambda)^{N+1}}{(N+1)!}  (\# \Lambda)^{\frac{1}{2}} \|\widehat{u}\|_{\ell^2}.
    \end{equation*}
    Stirling's approximation formula for factorials gives
    \begin{equation*}
        \|u\|_{L^\infty(B_r)}
        \lesssim \frac{(\# \Lambda)^{\frac{1}{2}}}{\sqrt{N+1}} \Bigl(\frac{e\cdot r\diam \Lambda}{N+1}\Bigr)^{N+1} \|\widehat{u}\|_{\ell^2}.
    \end{equation*}
    By Parseval's identity: $\int_{\mathbb{T}^d} |u(x)|^2 \dd x \sim_d \|\widehat{u}\|_{\ell^2}^2$.
    We thus conclude with
    \begin{equation*}
        m_d(\lambda,r) \lesssim_d \frac{\|u\|_{L^\infty(B_r)}}{\|\widehat{u}\|_{\ell^2}}.
        \qedhere
    \end{equation*}
\end{proof}

To minimize the right hand side of \eqref{eq::upperbound-via-vanishing}, we need to maximize $\Gamma_d(\Lambda)$ among all $\Lambda \subset \mathcal{S}_d(\lambda)$ with $\diam \Lambda \le 2R$ for any fixed $R>0$.
By Theorem~\ref{thm::vanishing}, we have
\begin{equation*}
    \Gamma_d(\Lambda) \gtrsim (\# \Lambda)^{\frac{1}{d-1}}.
\end{equation*}
Therefore, a good approximation of this problem is the maximization of $\# \Lambda$ among all subsets of bounded diameters.
This leads to the definition of the quantities
\begin{equation*}
    \mathcal{A}_d(\lambda,R) = \frac{\mathcal{N}_d(\lambda,R)^{\frac{1}{d-1}}}{R},
    \qquad
    \mathcal{N}_d(\lambda,R) = \max\{\# \Lambda \mid \Lambda \subset \mathcal{S}_d(\lambda),\; \diam \Lambda \le 2R\}.
\end{equation*}
In other words, $\mathcal{N}_d(\lambda,R)$ is the maximum possible number of points one can find in a subset of $\mathcal{S}_d(\lambda)$ of diameter $\le 2R$, and $\mathcal{A}_d(\lambda,R)^{d-1}$ heuristically describes, up to some multiplicative constants, the maximum density of integer points in a spherical cap of radius $\sim R$. 
Clearly when $R = \pi \lambda$ we have $\mathcal{N}_d(\lambda,R) = \mathcal{N}_d(\lambda)$; in this case we will also denote for simplicity $\mathcal{A}_d(\lambda,R) = \mathcal{A}_d(\lambda)$.

\begin{cor}
    \label{cor::upperbound-high-dimension}
    If $d\ge 2$ and $r \ll_d \mathcal{A}_d(\lambda,R)$ for some $R \in (0,\pi\lambda]$, then
    \begin{equation}
    \label{eq::upperbound-compare-density}
        m_d(\lambda,r) \lesssim \Bigl( C_d r \mathcal{A}_d(\lambda,R)^{-1} \Bigr)^{C_d \mathcal{N}_d(\lambda,R)^{\frac{1}{d-1}}}.
    \end{equation}

    In particular, if $d \ge 5$ and $r \ll_d \lambda^{-\frac{1}{d-1}}$, then
    \begin{equation*}
        m_d(\lambda,r) \lesssim_d \bigl( C'_d r \lambda^{\frac{1}{d-1}} \bigr)^{C_d \lambda^{\frac{d-2}{d-1}}}.
    \end{equation*}
\end{cor}
\begin{proof}
    For $R\in(0,\pi\lambda]$ choose $\Lambda \subset \mathcal{S}_d(\lambda)$ with $\diam \Lambda \le 2R$  and $\# \Lambda = \mathcal{N}_d(\lambda,R)$.
    If $r \ll_d \mathcal{A}_d(\lambda,R)$, then $r\diam \Lambda \ll_d \Gamma_d(\Lambda)$.
    Therefore, we apply \eqref{eq::upperbound-via-vanishing} and Theorem~\ref{thm::vanishing-general} for
    \begin{equation*}
        m_d(\lambda,r) \lesssim (\# \Lambda)^{\frac{d-2}{2(d-1)}} \Bigl( 2e R\Bigl(C_d (\# \Lambda)^{\frac{1}{d-1}}\Bigr)^{-1}\Bigr)^{C_d (\# \Lambda)^{\frac{1}{d-1}}}
    \end{equation*}
    To prove \eqref{eq::upperbound-compare-density}, it remains to absorb the power of $\# \Lambda$ by its exponential and invoke the definition of $\mathcal{A}_d(\lambda,R)$.
    To prove the particular cases where $d\ge 5$, it suffices to plug in $R = \pi \lambda$ and use
    \begin{equation*}
        \mathcal{A}_d(\lambda) \sim_d \frac{\lambda^{\frac{d-2}{d-1}}}{\lambda} = \lambda^{-\frac{1}{d-1}}.
        \qedhere
    \end{equation*}
\end{proof}

\begin{rem}
\label{rem::main-thm-vs-equidist}
    By \cite[App.~A]{BourgainRudnick} and \cite[Cor.~1.9]{Sardari2019optimal}, if $d \ge 5$ and $R \ge \lambda^{\frac{1}{2}+\epsilon}$, then
    \begin{equation*}
        \mathcal{A}_d(\lambda,R) 
        \sim_d \frac{1}{R} \Bigl(\frac{R^{d-1}}{\lambda}\Bigr)^{\frac{1}{d-1}} 
        \sim_d \lambda^{-\frac{1}{d-1}}.
    \end{equation*}
    Therefore, when $d \ge 5$, to obtain estimates when $r \gtrsim_d \lambda^{-\frac{1}{d-1}}$, one needs to consider $R$ for which $\mathcal{A}_d(\lambda,R) \gg \mathcal{A}_d(\lambda)$ (so that the equidistribution $\mathcal{A}_d(\lambda,R) \sim \mathcal{A}_d(\lambda)$ is violated).
    As observed above, this may only happen when $R$ is very small (at least smaller than $\lambda^{-\frac{1}{2}}$).
    
    Similar argument can be made when $d =3,4$ as well; in these cases explicit equidistribution results on small areas of spheres have been obtained by \cite{Sardari2019optimal} (for $d=4$) and \cite{Duke-SP1990representation,Humphries-R2022optimal} (for $d=3$).
\end{rem}

\begin{rem}
    \label{rem::low-dim-lattice-distri}
    Typically the violation of the equidistribution is a consequence of lower dimensional distribution of lattice points, as already remarked in \cite[Prop.~1.4]{BourgainRudnick}.
    In \S\ref{sec::example-degenerate} we will study examples related to this phenomena.
\end{rem}

\subsection{Alternative proof} 

We provide here a slightly different proof of Theorem~\ref{thm::upperbound-via-vanishing}.
For simplicity we shall only deal with the case $\Lambda = \mathcal{S}_d(\lambda)$, but the analysis applies for general situations.

Instead of arguing as in the previous subsection by first constructing functions with maximal orders of pointwise vanishing and then estimating their local $L^2$-norms through Taylor expansions, we consider directly the quadratic form associated with the local $L^2$-norm.
Namely, for $\chi \in C_0^\infty(\mathbb{R}^d)$ write
\begin{align*}
    \frac{1}{r^d} \int &\chi \Bigl( \frac{x}{r} \Bigr) |u(x)|^2 \dd x 
    = \sum_{k,\ell \in \mathcal{S}_d(\lambda)} \widehat{\chi} \bigl(r (k-\ell)\bigr) \widehat{u}_k \overline{\widehat{u}_\ell} \\
    & = \sum_{|\alpha|\le N} \frac{(ir)^{|\alpha|} J_\alpha}{\alpha!}  \sum_{k,\ell \in \mathcal{S}_d(\lambda)} (k-\ell)^\alpha \widehat{u}_k \overline{\widehat{u}_\ell} 
    + \frac{\mathcal{O} (C_d r)^{N+1}}{(N+1)!} \sum_{k,\ell \in \mathcal{S}_d(\lambda)} |k-\ell|^{N+1} |\widehat{u}_k \overline{\widehat{u}_\ell}|,
\end{align*}
where 
\begin{equation*}
    J_\alpha = \int \chi(x) x^\alpha \dd x.
\end{equation*}
By Theorem~\ref{thm::vanishing}, when
\begin{equation*}
    \binom{N+d-1}{d-1} + \binom{N+d-2}{d-1} < \mathcal{N}_d(\lambda),
\end{equation*}
there exists $u \in \mathcal{E}_d(\lambda) \backslash \{0\}$ such that \eqref{eq::hirondelle} holds.
Therefore,
\begin{equation*}
    (\widehat{u}_k)_{k \in \mathcal{S}_d(\lambda)} \in \bigcap_{|\alpha|\le N} \ker A_\alpha,
\end{equation*}
where $A_\alpha = \bigl( (k-\ell)^\alpha \bigr)_{k,\ell \in \mathcal{S}_d(\lambda)}$.
Consequently, for this $u$,
\begin{equation*}
    \frac{1}{r^d} \int \chi \Bigl( \frac{x}{r} \Bigr) |u(x)|^2 \dd x 
    \lesssim \frac{(C_d r \lambda)^{N+1}}{(N+1)!} \sum_{k,\ell \in \mathcal{S}_d(\lambda)} |\widehat{u}_k \overline{\widehat{u}_\ell}|
    \lesssim \frac{(C_d r \lambda)^{N+1}}{(N+1)!} \mathcal{N}_d(\lambda) \|u\|_{L^2}^2.
\end{equation*}
It remains to apply Stirling's approximation as in the proof of Theorem~\ref{thm::upperbound-via-vanishing}.

\subsection{Examples with lower dimensional Fourier support}
\label{sec::example-degenerate}

We construct examples whose Fourier transform is supported on the intersection of the sphere and a hyperplane. This class of examples seems optimal in some range of the parameters $r$ and $\lambda$.

\begin{rem}
    In view of \eqref{eq::upperbound-via-vanishing}, the goal is to maximize the possible order of vanishing for eigenfunctions whose Fourier supports lie in lattice sets with prescribed diameter bounds.
    As noted in Remarks~\ref{rem::main-thm-vs-equidist} and~\ref{rem::low-dim-lattice-distri}, this requires breaking the equidistribution of lattice points, which typically occurs when the Fourier support is confined to lower-dimensional affine subspaces.
    To achieve the maximal vanishing order, one must maximize the dimension of such subspaces, and hence restrict attention to hyperplanes.
\end{rem}

\begin{thm}
    \label{thm::QL-upperbound-high-dim}
    For $d\ge 4$ and $r \in (0,1)$ we have
    \begin{equation*}
        m_d(r) \lesssim \exp\bigl\{-C_d r^{3-d} \bigr\}.
    \end{equation*}
\end{thm}

\begin{proof}
    Applying \eqref{eq::upperbound-via-vanishing} on $\mathbb{T}^{d-1}$ with $\Lambda = \mathcal{S}_{d-1}(R)$ where $R>0$ is to be found, we obtain an eigenfunction $u' \in \mathcal{E}_{d-1}(R)$ on $\mathbb{T}^{d-1}$ with $\|u'\|_{L^2(\mathbb{T}^{d-1})} = 1$, such that
    \begin{equation*}
        \fint_{B_r^{d-1}} |u'(x')|^2 \dd x'
        \lesssim \bigl( C_d r R^{\frac{1}{d-2}}\bigr)^{C_d R^{\frac{d-3}{d-2}}},
    \end{equation*}
    where we denote by $B_r^j$ the $r$-ball in $\mathbb{R}^j$ to avoid ambiguity.
    The right hand side is a function of $R$ and attains its minimum when $r R^{\frac{1}{d-2}} \sim_d 1$.
    
    We claim that, for all $K \gg 1$, there exists $R\sim_d (Kr)^{2-d}$ with $\mathcal{N}_{d-1}(R) \gtrsim R^{d-3}$.
    Indeed, let $\varrho = (Kr)^{2-d}$ and note that $\mathcal{E}_{d-1}(R) \ne \emptyset$ whenever $R^2 \in \mathbb{N}$.
    By the area estimate,
    \begin{equation*}
        \max_{R \in (\varrho,2\varrho)} \mathcal{N}_{d-1}(R)
        \ge \frac{\#\{k \in \mathbb{Z}^{d-1} : \varrho < |k| < 2\varrho\}}{\#\{R \in (\varrho,2\varrho):\mathcal{E}_{d-1}(R)\ne\emptyset\}}
        \sim \frac{\varrho^{d-1}}{\varrho^2} = \varrho^{d-3},
    \end{equation*}
    implying the existence of such $R \in (\varrho,2\varrho)$.
    Plugging in such $R$ yields
    \begin{equation*}
        \fint_{B_r^{d-1}} |u'(x')|^2 \dd x'
        \lesssim \exp\bigl\{-C'_d r^{3-d}\bigr\}.
    \end{equation*}
    Finally, let $\lambda_n = \sqrt{R^2 + n^2}$ and construct $u_n \in \mathcal{E}_d(\lambda_n)$ by setting
    \begin{equation*}
        u_n(x) = e^{inx_d} u'(x').
    \end{equation*}
    Then $\|u_n\|_{L^2} = 1$, and
    \begin{equation*}
        m_d(\lambda_n,r) 
        \le \fint_{B_r^d} |u_n(x)|^2 \dd x
        \lesssim \fint_{B_r^{d-1}} |u'(x')|^2 \dd x'
        \lesssim \exp\bigl\{-C'_d r^{3-d}\bigr\}.
    \end{equation*}
    We conclude by passing $n\to\infty$ using \eqref{eq::QL-relation}.
\end{proof}

The construction when $d=3$ requires more involved number theoretical results.

\begin{thm}
    \label{thm::QL-upperbound-3D}
    For $d=3$ and $r \in (0,1)$, we have
    \begin{equation*}
        m_3(r) \lesssim \exp\Bigl\{- \exp\Bigl(\frac{C\ln(1/r)}{\ln\ln(1/r)}\Bigr) \Bigr\}.
    \end{equation*}
\end{thm}
\begin{proof}
    Our proof is inspired by Wigert's estimate of the divisor function \cite{Wigert1907ordre}.
    Let $\mathcal{G}$ be the set of all prime numbers which are $\equiv 1 \pmod{4}$.
    For all $m \in \mathbb{N}$, let 
    \begin{equation*}
        \mathcal{G}_m = \mathcal{G} \cap [0,m],
        \quad
        P_m = \prod_{p \in \mathcal{G}_m} p,
        \quad
        \varpi_m = \# \mathcal{G}_m.
    \end{equation*}
    Dirichlet's theorem on arithmetic progressions states that $\varpi_m \sim \frac{m}{\ln m}$.
    
    From now on, fix $m \ge 0$ to be such that
    \begin{equation*}
        \sqrt{P_{m-1}} \le \frac{1}{r} \le \sqrt{P_m}.
    \end{equation*}
    By \cite{McCurley-Kevin1984prime} (see also \cite{BMOR2018primes} for finer estimates), we have
    \begin{equation*}
        m \sim \ln P_m \sim \ln\Bigl( \frac{1}{r} \Bigr).
    \end{equation*}
    Therefore,
    \begin{equation*}
        \frac{1}{r} \le \sqrt{P_m} \le \frac{1}{r} \frac{\sqrt{P_m}}{\sqrt{P_{m-1}}} \le \frac{\sqrt{m}}{r} \lesssim \frac{\sqrt{\ln(1/r)}}{r}.
    \end{equation*}
    Next, due to the sum of two squares theorem by Jacobi \cite{Jacobi1829fundamenta}, we have
    \begin{equation*}
        \mathcal{N}_2(\lambda) = 4(d_1(\lambda^2)-d_3(\lambda^2))
    \end{equation*}
    for all $\lambda\ge 0$, where $d_j(n) = \#\{k\in \mathbb{N}: k|n, k\equiv j \pmod{4}\}$.
    By our construction $d_3(P_m)=0$, therefore
    \begin{equation*}
        \mathcal{N}_2(\sqrt{P_m}) = 4 d_1(P_m) = 4 \cdot 2^{\varpi_n}
        = e^{\frac{\gamma_m m}{\ln m}}, \quad \gamma_m \sim 1.
    \end{equation*}

    By Theorems~\ref{thm::upperbound-via-vanishing} and~\ref{thm::vanishing-order-upperbound}, there exists $u' \in \mathcal{E}_2(\sqrt{P_m})$ with $\|u'\|_{L^2(\mathbb{T}^2)} = 1$, such that
    \begin{align*}
        \fint_{B_r^2} |u'(x')|^2 \dd x'
        & \lesssim \Bigl( \frac{C r\sqrt{P_m}}{\mathcal{N}_2(\sqrt{P_m})} \Bigr)^{C \mathcal{N}_2(\sqrt{P_m})}
        \lesssim \Bigl( C e^{-\frac{\gamma_m m}{\ln m}} \Bigr)^{C\exp\{\frac{\gamma_m m}{\ln m}\}} \\
        & = \exp\biggl\{C'\biggl( \ln M - \frac{\gamma_m \ln m}{\ln \ln m}\biggr) \exp\biggl\{\frac{\gamma_m \ln m}{\ln\ln m}\biggr\} \biggr\} \\
        & \le \exp\biggl\{- \exp\biggl\{\frac{C''\ln(1/r)}{\ln\ln(1/r)}\biggr\} \biggr\}.
    \end{align*}
    Let $\lambda_n = \sqrt{P_m + n^2}$ and construct $u_n \in \mathcal{E}_3(\lambda_n)$ by setting
    \begin{equation*}
        u_n(x) = e^{inx_3} u'(x_1,x_2),
    \end{equation*}
    then $\|u_n\|_{L^2(\mathbb{T}^3)} = 1$, and
    \begin{align*}
        m_3(\lambda_n,r)
        & \le \fint_{B_r^3} |u_n(x)|^2 \dd x
        \lesssim \fint_{B_r^2} |u'(x')|^2 \dd x'
        \lesssim \exp\biggl\{- \exp\biggl\{\frac{C''\ln(1/r)}{\ln\ln(1/r)}\biggr\} \biggr\}.
    \end{align*}
    We conclude by passing $n\to\infty$ and using \eqref{eq::QL-relation}.
\end{proof}

\section{Lower bounds}
\label{sec::lowerbounds}

By Lemmas~\ref{lem::nazarov} and~\ref{lem::FM} (Nazarov--Turán lemmas), the following lower bound holds
\begin{thm}
    \label{thm::m-lower}
    For all $d,\lambda,r$, there holds
    \begin{equation}
    \label{eq::lower-bound-general-by-Nazarov}
        m_{d}(\lambda,r) \ge (C_d r)^{2d\min\{\mathcal{N}_d(\lambda)-1,\lambda\}}.
    \end{equation}
\end{thm}
In the following we provide some improvements of this general lower bounds when $r$ is not too small. 
We follow the footsteps of Bourgain--Rudnick \cite{BourgainRudnick}, using cluster structures of integer points on spheres.
Related results are recalled in \S\ref{sec::cluster}.
The following lemma reduces the estimate of general eigenfunctions to those with localized Fourier coefficients.

\begin{df}
    \label{def::graph}
    For $d \ge 2$, $\varrho>0$ and $V\subset \mathbb{R}^d$, denote by $G(V,\varrho)$ the graph formed by connecting point pairs $V$ with distances $< \varrho$.
    The disjoint ensemble of vertex sets of the connected components of $G(\mathcal{S}_d(\lambda),\varrho)$ is denoted by $\Omega = \{\Omega_\alpha\}_\alpha$.
    Clearly $\Omega$ gives a partition of $V$, i.e., $V = \bigsqcup_\alpha \Omega_\alpha$.
\end{df}

\begin{lem}
    \label{lem::decomposition}
    Let $V \subset \mathbb{Z}^d$ be a finite set and let $\Omega = (\Omega_\alpha)_\alpha$ be the partition of $V$ given by $G(V,\varrho)$ where $\varrho > 0$, as according to Definition~\ref{def::graph}.
    For all $\alpha$, denote by $\Pi_\alpha$ the Fourier projection onto $\Omega_\alpha$.
    Then, for all $u \in C^\infty(\mathbb{T}^d)$ with Fourier coefficients supported on $V$ and all $\chi \in C_c^\infty(\mathbb{R}^d)$, there holds
    \begin{equation}
        \label{eq::decomposition}
        \biggl| \frac{1}{r^d} \int \chi\Bigl(\frac{x}{r}\Bigr) |u(x)|^2 \dd x
        - \sum_{\alpha} \frac{1}{r^d} \int \chi\Bigl(\frac{x}{r}\Bigr) |\Pi_\alpha u(x)|^2 \dd x \biggr| \le \frac{1}{2} \# V \sup_{|\xi|\ge r\varrho} |\widehat{\chi}(\xi)| \cdot\|u\|_{L^2}^2.
    \end{equation}
\end{lem}
\begin{proof}
    Using the rapid decay of $\widehat{\chi}$, we bound the sum of cross terms as follows:
    \begin{align*}
        \sum_{\alpha\ne\beta} \frac{1}{r^d} \biggl| \int \chi\Bigl(\frac{x}{r}\Bigr) \Pi_\alpha u(x) \overline{\Pi_\beta u(x)} \dd x \biggr|
        & \le \sum_{k,\ell \in V} \bm{1}_{|k-\ell|>\varrho} \bigl|\widehat{\chi}\bigl(r(k-\ell)\bigr)\bigr| \cdot |\widehat{u}(k)| \cdot |\overline{\widehat{u}(\ell)}| \\
        & \le  \frac{1}{2} \sum_{k,\ell \in V} 
        \sup_{|\xi|\ge r\varrho} |\widehat{\chi}(\xi)| \bigl( |\widehat{u}(k)|^2 + |\overline{\widehat{u}(\ell)}|^2 \bigr) \\
        & = \frac{1}{2} \# V \sup_{|\xi|\ge r\varrho} |\widehat{\chi}(\xi)| \cdot\|u\|_{L^2}^2.
        \qedhere
    \end{align*}
\end{proof}

\subsection{The case of dimension 2}
\label{sec::lower-bounds-2D}

Let us recall the lemma by Jarník \cite{Jarnik1926} on the distribution of integer points on circles, which is stated in Lemma~\ref{lem::jarnik}.

\begin{thm}
    \label{2dthm}
    If $r > \lambda^{-\frac{1}{3}+\epsilon}$ where $\epsilon \in (0,\frac{1}{3})$, then
    \begin{equation*}
        m_2(\lambda,r) \gtrsim_\epsilon r^2.
    \end{equation*}
    Consequently, there holds
    \begin{equation*}
        m_2(r) \gtrsim r^2.
    \end{equation*}
\end{thm}
\begin{proof}
    By Lemma~\ref{lem::jarnik}, let $\Omega = (\Omega_\alpha)_\alpha$ be the partition given by $G(\mathcal{S}_2(\lambda),\frac{\sqrt{2}}{2}\lambda^{\frac{1}{3}})$, then $\# \Omega_\alpha \le 2$ for all $\alpha$.
    Indeed, if $\#\Omega_\alpha \ge 3$ for some $\alpha$, then we may find three distinct but adjacent integer points on an arc of length $<\sqrt{2}\lambda^{\frac{1}{3}}$, contradicting Lemma~\ref{lem::jarnik}.
    Let $\chi \in C_c^\infty(\mathbb{R}^2)$ be radial, satisfying $\widehat{\chi}(0)=1$, $\bm{1}_{B_r} \gtrsim \chi \gtrsim \bm{1}_{B_{r/2}}$, and apply Lemma~\ref{lem::decomposition}.
    For the estimate of $\Pi_\alpha u$, clearly the only difficulty cases are those $\alpha$ with $\# \Omega_\alpha = 2$.
    For such $\alpha$, write $\Omega_\alpha = \{k,\ell\}$. Then
    \begin{equation*}
        \frac{1}{r^2} \int \chi\Bigl(\frac{x}{r}\Bigr) |\Pi_\alpha u(x)|^2 \dd x
        = |\widehat{u}_k|^2 + |\widehat{u}_\ell|^2 + \widehat{\chi}\bigl(r(k-\ell)\bigr) \widehat{u}_k \overline{\widehat{u}_\ell} + \widehat{\chi}\bigl(r(\ell-k)\bigr) \widehat{u}_\ell \overline{\widehat{u}_k},
    \end{equation*}
    where $\chi \in C_c^\infty(\mathbb{R}^2)$ is radial with $\widehat{\chi}(0)=1$.
    The right-hand side is a quadratic form $Q$ in $(\widehat{u}_k,\widehat{u}_\ell)$.
    Therefore, matters reduce to estimating the smallest eigenvalue of a $2\times 2$ matrix; since $1 - \widehat{\chi}(r(k-\ell)) \gtrsim r^2$, we see that the smallest eigenvalue of $Q$ is $\gtrsim r^2$.
    Therefore
    \begin{equation*}
        \frac{1}{r^2} \int \chi\Bigl(\frac{x}{r}\Bigr) |\Pi_\alpha u(x)|^2 \dd x \gtrsim r^2 \|\Pi_\alpha u\|_{L^2}^2.
    \end{equation*}
    We conclude by summing over $\alpha$ and using orthogonality.
\end{proof}

To go beyond the two previous theorems, we will resort to the a result of Cilleruelo and Cordoba \cite{CillerueloCordoba} (Lemma~\ref{lem::CC}) which generalizes Jarník's result (Lemma~\ref{lem::jarnik}).

\begin{thm}
    \label{thm::2d-lowerbound-CC}
    If $m \ge 1$ and $r > \lambda^{- \frac{1}{2} + \delta(m) + \epsilon}$ where $\epsilon \in \bigl(0,\frac{1}{2}-\delta(m)\bigr)$, then
    \begin{equation*}
        m_2(\lambda,r) \gtrsim_{\epsilon,m} r^{4(m-1)}.
    \end{equation*}
\end{thm}
\begin{proof}
    By Lemma~\ref{lem::jarnik-connes}, let $\Omega = (\Omega_\alpha)_\alpha$ be the partition given by $G(\mathcal{S}_2(\lambda),\frac{\sqrt{2}}{m}\lambda^{\frac{1}{2}-\delta(m)})$, then $\# \Omega_\alpha \le m$ for all $\alpha$.
    Indeed, if $\#\Omega_\alpha \ge m+1$ for some $\alpha$, then we may find $m+1$ distinct but adjacent integer points on an arc of length $<\sqrt{2}\lambda^{\frac{1}{2}-\delta(m)}$, contradicting Lemma~\ref{lem::jarnik-connes}.
    Let $\chi \in C_c^\infty(\mathbb{R}^2)$ be such that $\bm{1}_{B_r} \gtrsim \chi \gtrsim \bm{1}_{B_{r/2}}$ and apply Lemma~\ref{lem::decomposition}.
    Now, apply \eqref{eq::nazarov-p} to each $\Pi_\alpha u$ yields
    \begin{equation*}
        \frac{1}{r^2} \int \chi\Bigl(\frac{x}{r}\Bigr) |\Pi_\alpha u(x)|^2 \dd x
        \gtrsim \frac{1}{r^2} \int_{B_r}|\Pi_\alpha u(x)|^2 \dd x \gtrsim (C r)^{4(m-1)} \|\Pi_\alpha u\|_{L^2}^2.
    \end{equation*}
    We conclude by summing over $\alpha$ and using orthogonality.
\end{proof}

\begin{rem}
    \label{rem::ineger-point-circle}
    An interesting question is whether the method of proof of the theorem above can be pushed further. Instead of Nazarov's theorem, it might be possible to obtain a more precise bound by considering the rather explicit quadratic form restricted to $m$ points --- however, we were not able to carry out this plan. Another possible improvement is of a number-theoretic nature, namely the conjecture of Cilleruelo--Granville \cite{CillerueloGranville} that the number of integer points on an arc of length $\lambda^{1-\epsilon}$ of the circle of radius $\lambda$ is bounded uniformly if $\epsilon>0$.
\end{rem}

\subsection{The case of higher dimensions}
\label{sec::lowerbound-high-dim}

The same spirit of the proof of Theorem~\ref{thm::2d-lowerbound-CC} naturally extends to higher dimensions.
Our analysis relies on mathematical induction and requires the consideration of general ellipsoids.
We advice the readers to refer to \S\ref{sec::cluster} for concepts (such as the radius and the ratio) and results on the cluster structure of integer points on ellipsoids.

Since $m_d(r)$ decays exponentially as $r\to 0$ when $d\ge 3$, as we have seen, a careful estimate of the remainder is required for sharper results.
Therefore, another important ingredient is using a multidimensional Beurling--Malliavin theorem to choose a cutoff function $\chi$ whose Fourier transform exhibits a sub-exponential decay.

\begin{lem}
    \label{lem::Beurling--Malliavin}
    For all $d\ge 1$, $\sigma>0$ and $\gamma \in (0,1)$, there exists a nonnegative function $\chi \in C_c^\infty(\mathbb{R}^d)$ supported in $B_\sigma$, such that $\chi(0) > 0$ and 
    \begin{equation*}
        |\widehat{\chi}(\xi)| \le e^{-|\xi|^\gamma},
        \quad
        \forall \xi \in \mathbb{R}^d.
    \end{equation*}
\end{lem}
\begin{proof}
    First, we claim that for all $\kappa > 0$, there exists a nonzero, radial, and real-valued function $f \in C_c^\infty(\mathbb{R}^d)$ with $\supp f \subset B_{\sigma/4}$ and satisfies
    \begin{equation*}
        |\widehat{f}(\xi)| \le e^{- \kappa |\xi|^\gamma}.
    \end{equation*}
    Indeed, by the classical Beurling--Malliavin theorem \cite{BeurlingMalliavin1962fourier}, for all $\kappa'>0$, there exists a real-valued $\varphi \in C_c^\infty(\mathbb{R}^d)$ which was supported in $B^1_{\sigma/4}$ and satisfies $\widehat{\varphi}(\xi) \le e^{-\kappa'|\xi|^\gamma}$.
    Let
    \begin{equation*}
        f(\xi) = \prod_{j=1}^d \varphi(\xi^j)
    \end{equation*}
    and let $\kappa'$ be sufficiently large.
    Next let $g = f \ast f$, then $g$ is real valued and satisfies $\supp g \subset B_{\sigma/2}$ and $g(0) = \|f\|_{L^2}^2 > 0$.
    Clearly
    \begin{equation*}
        |\widehat{g}(\xi)| = |\widehat{f}(\xi)|^2 \le e^{-2\kappa |\xi|^\gamma}.
    \end{equation*}
    Choose $\chi = \beta |g|^2$ where $\beta>0$ will be fixed later.
    Then we have
    \begin{equation*}
        |\widehat{\chi}(\xi)|
        \le \beta \int |\widehat{g}(\xi-\eta)| \cdot |\widehat{g}(\eta)| \dd \eta
        \le \beta \int e^{-2\kappa(|\xi-\eta|^\gamma+|\eta|^\gamma)} \dd \eta.
    \end{equation*}
    Note that, by the triangular inequality if $|\xi-\eta|\le \frac{1}{2}|\xi|$, then $|\eta|\ge \frac{1}{2}|\xi|$.
    Therefore
    \begin{align*}
        \int e^{-2\kappa(|\xi-\eta|^\gamma +|\eta|^\gamma)} \dd \eta
        & \le \int_{|\xi-\eta|\ge \frac{1}{2}|\xi|} e^{2\kappa(\frac{1}{2^\gamma} |\xi|^\gamma + |\eta|^\gamma)} \dd \eta
        + \int_{|\xi-\eta|\le \frac{1}{2}|\xi|} e^{2\kappa(|\xi-\eta|^\gamma + \frac{1}{2^\gamma} |\xi|^\gamma)} \dd \eta \\
        & \le e^{\frac{2\kappa}{2^\gamma}|\xi|^\gamma} \biggl( \int_{|\xi-\eta|\ge \frac{1}{2}|\xi|} e^{2\kappa |\eta|^\gamma} \dd \eta + \int_{|\xi-\eta|\le \frac{1}{2}|\xi|} e^{2\kappa|\xi-\eta|^\gamma} \dd \eta \biggr) \\
        & \le 2 e^{\frac{2\kappa}{2^\gamma}|\xi|^\gamma} \biggl( \int e^{2\kappa |\eta|^\gamma} \dd \eta \biggr).
    \end{align*}
    We conclude by first choosing a sufficiently large $\kappa$ and then a sufficiently small $\beta$.
\end{proof}

\begin{rem}
    By applying \cite[Thm.~1]{Vasilyev2024BM}, we may further assume that $\chi$ is radial.
\end{rem}

\begin{thm}
    \label{thm::lower-bound-QL}
    Let $\gamma,\delta \in (0,1)$ and let $D>0$.
    For $d \ge 3$, define recursively
    \begin{equation*}
        \phi_{d+1}(r) = \Bigl(\frac{\phi_d(r)}{ r^\gamma}\Bigr)^{\frac{(d+1)!}{2\gamma}},
        \qquad
        \phi_3(r) = \exp\Bigl\{ \frac{D \ln(1/r)}{\ln \ln (1/r)} \Bigr\}.
    \end{equation*}
    If $\nu \ge 1$ and $D \gg_{\nu} 1$, then for any ellipsoid $E$ in $\mathbb{R}^d$ with radius
    \begin{equation*}
        \lambda > \phi_{d+1}(r) / (d+1)
    \end{equation*}
    and ratio $\le \nu$, and for all $u \in C^\infty(\mathbb{T}^d)$ with Fourier coefficients supported on $E$,
    \begin{equation*}
        \fint_{B_r} |u(x)|^2 \dd x \gtrsim_{d,\nu} (Cr)^{\phi_d(r)} \|u\|_{L^2}.
    \end{equation*}
    As consequences, the following two statements hold:
    \begin{itemize}
        \item If $d \ge 3$, $D \gg 1$, and $\lambda > \phi_d(r)$, then
        \begin{equation*}
            m_d(\lambda,r) \gtrsim_{d} (Cr)^{\phi_d(r)}.
        \end{equation*}
        \item If $d \ge 3$, $D \gg 1$, and $r \in (0,1)$, then
        \begin{equation*}
            m_d(r) \gtrsim_{d} (Cr)^{\phi_d(r)}.
        \end{equation*}
    \end{itemize}
\end{thm}
\begin{proof}
    We proceed using mathematical induction.
    We start by noticing that the proof Theorem~\ref{2dthm} implies that $m_{2,\nu}(\lambda,r) \gtrsim_{\nu} r^2$ when $r >_\nu \lambda^{-\frac{1}{3}+\epsilon}$.
    In fact it suffices to use in the proof the more general Lemma~\ref{lem::jarnik-connes} (by Connes) instead of Lemma~\ref{lem::jarnik} (by Jarník).

    For $d \ge 3$, let $\Omega$ be the partition of $E\cap \mathbb{Z}^d$ given by Lemma~\ref{lem::connes}.
    Since $\Omega_\alpha$ lies on an affine hyperplane, it lies on a $d-2$ dimensional ellipsoid, say $E_\alpha$.
    Clearly $E_\alpha$ is of radius $R_\alpha \le \lambda$ and of ratio $\le \nu$.
    To see the latter, it suffices to using the scaling along the principal directions that transforms $E$ and $E_\alpha$ to spheres and noticing that the ratios between scaling constants at different principal directions are bounded by $\nu$.

    If $R_\alpha > \phi_d(r)/d$, then the induction hypothesis applies.
    Indeed, let $X_\alpha$ be a hypersurface on which $E_\alpha$ lives and consider the projection $P : X_\alpha \to Y \simeq \mathbb{R}^{d-1}$ where $Y$ is an axis plane.
    We may choose $Y$ such that $P$ is bijective and satisfies $\|P^{-1}\| \lesssim c_d$ for some $c_d>0$ that only depends on $d$.
    This projection sends $E_\alpha$ to an ellipsoid on $\mathbb{R}^{d-1}$ with radius $R_\alpha' \sim_d R_\alpha$ and ratio $\le c_d \nu$.
    It also sends (via the pushforward it defines) $\Pi_\alpha u$ to a trigonometric function $v(x') = \Pi_\alpha u(P^{-1} x')$ on $\mathbb{R}^{d-1}$ whose Fourier coefficients lie on $P E_\alpha$.
    The induction hypothesis implies, if $D \gg_{d,\nu} 1$, then
    \begin{align*}
        \fint_{B_r^d}|\Pi_\alpha u(x)|^2 \dd x
        \gtrsim_d
        \fint_{B_r^{d-1}} |v(x')|^2 \dd x'
         & \gtrsim_{d,\nu} (Cr)^{\phi_{d-1}(r)} \|v\|_{L^2(\mathbb{T}^{d-1})}^2 \\
         & \gtrsim_{d,\nu} (Cr)^{\phi_{d-1}(r)} \|\Pi_\alpha u\|_{L^2(\mathbb{T}^d)}^2.
    \end{align*}

    If $R_\alpha \le \phi_d(r)/d$, then by \eqref{eq::nazarov-p} and \eqref{eq::FM-finer-ineq},
    \begin{equation*}
        \fint_{B_r} |\Pi_\alpha u(x)|^2 \dd x
        \gtrsim (Cr)^{d\min\{ R_\alpha,\#\Omega_\alpha \}} \|\Pi_\alpha u\|_{L^2(\mathbb{T}^d)}^2.
    \end{equation*}
    Then we bound $\min\{R_\alpha,\#\Omega_\alpha\}$ as follows:
    \begin{itemize}
        \item When $d=3$ we use \cite[Lem~3.2]{Sheffer2015} for
        \begin{equation*}
            \min\{R_\alpha,\#\Omega_\alpha\} \le \#\Omega_\alpha
                  \le \exp\bigl\{\frac{D\ln(1/r)}{\ln\ln (1/r)}\bigr\}
        \end{equation*}
        \item When $d \ge 4$, we use
        \begin{equation*}
            \min\{ R_\alpha,\#\Omega_\alpha\} \le R_\alpha \le \phi_d(r).
        \end{equation*}
    \end{itemize}
    Therefore, by the definition $\phi_d$, in both situation we obtain
    \begin{equation*}
        \fint_{B_r} |\Pi_\alpha u(x)|^2 \dd x
        \gtrsim_{d,\nu} (Cr)^{\phi_d(r)} \|\Pi_\alpha u\|_{L^2(\mathbb{T}^d)}^2.
    \end{equation*}
    Choose a bump function $\chi \in C_c^\infty(\mathbb{R}^d)$ according to Lemma~\ref{lem::Beurling--Malliavin} such that
    \begin{equation*}
        |\widehat{\chi}(\xi)| \le \exp\{-|d\xi|^{\gamma+\epsilon}\},
    \end{equation*}
    where $\xi \in (0,1-\gamma)$.
    For all $\alpha$, we have the same type of estimate:
    \begin{equation*}
        \frac{1}{r^d} \int \chi\Bigl(\frac{x}{r}\Bigr) |\Pi_\alpha u(x)|^2 \dd x
        \gtrsim_{d,\nu,N} (Cr)^{\phi_d(r)} \|\Pi_\alpha u\|_{L^2(\mathbb{T}^d)}^2.
    \end{equation*}
    Then apply Lemma~\ref{lem::decomposition}.
    Note that the right hand side of \eqref{eq::decomposition} is then
    \begin{equation*}
        \lesssim \exp\bigl\{-(dr)^{\gamma+\epsilon}(d\lambda)^{\frac{2(\gamma+\epsilon)}{(d+1)!}}\bigr\}.
    \end{equation*}
    In order to absorb this remainder, since $r \ll 1$, it suffices to require that
    \begin{equation*}
        \phi_d(r) <  r^{\gamma} (d\lambda)^{\frac{2\gamma}{(d+1)!}},
    \end{equation*}
    which leads to the condition $\lambda > \phi_{d+1}(r)/d$.
\end{proof}

\begin{rem}
    \label{rem::lower-bound-exponent-estimate}
    The proof works as well for a general bump function $\chi \in C_c^\infty(\mathbb{R}^d)$, but with worse (i.e., larger) exponents $\phi_d$.
    From the construction in the above theorem, we have the general formula
    \begin{equation*}
        \frac{(2\gamma)^d}{\prod_{j=1}^d j!} \log_r \phi_d(r)
        = \frac{(2\gamma)^3}{12} \log_r \phi_3(r) - \gamma \sum_{3\le k \le d-1} \frac{(2\gamma)^k}{\prod_{j=1}^{k} j!}.
    \end{equation*}
    Using $\phi_3(r) = r^{-\frac{D}{\ln\ln(1/r)}}$, we obtain that
    \begin{equation*}
        \phi_d(r) \lesssim_\epsilon r^{-h(d)-\epsilon}, \qquad
        h(d) = \gamma \sum_{3\le k \le d-1} (2\gamma)^{k-d} \prod_{k+1\le j \le d} j!.
    \end{equation*}
\end{rem}

\appendix

\section{Nazarov--Turán inequalities}
\label{sec::nazarov}

In this section we give several Nazarov--Turán type inequalities. 

\begin{lem}
    \label{lem::nazarov}
    For any trigonometric polynomial $f = \sum_{j=0}^n c_k e^{i k_j \cdot x}$ on $\mathbb{R}^d$ with $n+1$ terms, where $c_j \in \mathbb{C}$ and $k_j \in \mathbb{R}^d$, all Borel $E \subset B_R$ where $R > 0$, there holds
    \begin{equation}
    \label{eq::nazarov-infty}
        \|f\|_{L^\infty(E)} \gtrsim \Bigl(\frac{C}{R(1+R)^{d-1}} \frac{|E|}{(\diam E)^{d-1}}\Bigr)^n \|f\|_{L^\infty(B_R)}.
    \end{equation}
    Particularly, when $E = B_r \subset B_R$, there holds
    \begin{equation*}
        \|f\|_{L^\infty(B_r)} \gtrsim \Bigl(\frac{C r}{R(1+R)^{d-1}}\Bigr)^n \|f\|_{L^\infty(B_R)}.
    \end{equation*}

    Moreover, for all $p \in (0,\infty)$, there holds
    \begin{equation}
    \label{eq::nazarov-p}
    \biggl( \fint_E |f(x)|^p \dd x \biggr)^{\frac{1}{p}}
    \ge \Bigl(\frac{C |E|}{R(1+R)^{d-1}}\Bigr)^n \|f\|_{L^\infty(B_R)}.
    \end{equation}
\end{lem}
\begin{proof}
    The inequality \eqref{eq::nazarov-infty} was established by Nazarov \cite{Nazarov} when $d=1$.
    Now suppose that $d \ge 1$ and denote $r = \frac{1}{2} \diam E$ for simplicity.
    For all $x_0 \in B_R$, we have
    \begin{align*}
        |E|
         & = \int_{\mathbb{R}^d} \bm{1}_E(x_0+x) \dd x
        = \int_{\mathbb{S}^{d-1}} \int_0^\infty \bm{1}_E(x_0+\rho\theta) \rho^{d-1} \dd \rho \dd \theta \\
        & \lesssim 
        \begin{cases}
            \displaystyle
            r^{d-1} \sup_{\theta \in \mathbb{S}^{d-1}}  \int_0^\infty  \bm{1}_E(x_0+\rho\theta) \rho^{d-1} \dd \rho
            \lesssim r^{d-1} R^{d-1} \nu_{x_0}(E), & \dist(x_0,E) \gg r; \\
            \displaystyle
            \sup_{\theta \in \mathbb{S}^{d-1}}  \int_0^{Cr}  \bm{1}_E(x_0+\rho\theta) \rho^{d-1} \dd \rho
            \lesssim r^{d-1} \nu_{x_0}(E),  & \dist(x_0,E) \lesssim r,
        \end{cases}
    \end{align*}
    where
    \begin{equation*}
        \nu_{x_0}(E) = \sup_{\theta \in \mathbb{S}^{d-1}} \int_0^\infty \bm{1}_E(x_0+\rho\theta) \dd \rho.
    \end{equation*}
    Therefore, there exists a straight line $L$ passing through $x_0$ such that
    \begin{equation*}
        |E \cap L| \gtrsim \nu_{x_0}(E) \gtrsim \frac{1}{(1+R)^{d-1}} \frac{|E|}{r^{d-1}}.
    \end{equation*}
    Applying \eqref{eq::nazarov-infty} with $d=1$ on $L$ yields
    \begin{align*}
        \|f\|_{L^\infty(E)}
        \ge \|f\|_{L^\infty(E\cap L)}
        & \gtrsim \Bigl(\frac{C|E\cap L|}{|B_{2R}\cap L|}\Bigr)^n \|f\|_{L^\infty(B_{2R} \cap L)} \\
        & \gtrsim \Bigl(\frac{C'}{R(1+R)^{d-1}} \frac{|E|}{r^{d-1}}\Bigr)^n |f(x_0)|.
    \end{align*}
    Since $x_0$ is arbitrarily chosen, this proves \eqref{eq::nazarov-infty} for all dimensions.

    To prove \eqref{eq::nazarov-p}, we follow Nazarov \cite{Nazarov2000}.
    Let $\tilde{f}$ be the symmetric decreasing rearrangement of $|f|\bm{1}_{B_R}$ and let $\Phi$ be the profile of $\tilde{f}$, i.e., $\tilde{f}(x) = \Phi(|x|)$.
    Then reformulate \eqref{eq::nazarov-infty} as
    \begin{equation}
        \label{eq::nazarov-sym-dec}
        \inf\{\|f\|_{L^\infty(E)}: E \subset B_R, |E| = \tau\} = \Phi(\rho(\tau)) \gtrsim \Bigl(\frac{C\tau}{R(1+R)^{d-1}}\Bigr)^n \Phi(0),
    \end{equation}
    for all $\tau \in  [0, |B_R|]$, where $\rho(\tau)$ is given by the equation 
    \begin{equation*}
        v_d (R^d - \rho(\tau)^d) = \tau,
    \end{equation*}
    with $v_d$ being the volume of the unit ball in $\mathbb{R}^d$.
    Therefore, for all Borel $E\subset B_R$ and $p \in (0,\infty)$, there holds
    \begin{equation*}
        \lim_{p \to 0} \biggl( \fint_E |f(x)|^p \dd x \biggr)^{\frac{1}{p}} \\
        = \exp\Bigl\{\fint_E \ln |f(x)| \dd x \Bigr\}
        \ge \exp\Bigl\{\frac{s_d}{|E|} \int_{\rho(|E|)}^R [\ln \Phi(s)] s^{d-1}\dd s \Bigr\},
    \end{equation*}
    where $s_d$ is the area of the unit sphere $\mathbb{S}^{d-1}$.
    For the validity of the first limit above, see e.g., \cite[Chapter~3]{RudinBook}.
    By \eqref{eq::nazarov-sym-dec}, we have
    \begin{equation*}
        \int_{\rho(|E|)}^R [\ln \Phi(s)] s^{d-1} \dd s
        \ge \frac{|E|}{dv_d} \ln \Bigl[\frac{C^n}{R^n (1+R)^{n(d-1)}} \Phi(0)\Bigr] + n \int_{\rho(|E|)}^R [\ln \rho^{-1}(s)] s^{d-1} \dd s.
    \end{equation*}
    Changing the variables $s = \rho(\tau)$ and using $\rho^{d-1} \rho' \equiv -\frac{1}{d v_d}$ yields
    \begin{equation*}
        \int^R_{\rho(E)} [\ln \rho^{-1}(s)] \dd s
        = \frac{1}{d v_d} \int_0^{|E|} \ln \tau \dd \tau
        \ge \frac{1}{d v_d} |E| (\ln|E|-1).
    \end{equation*}
    Using $dv_d = s_d$, and using the Hölder inequality, we finally obtain that
    \begin{equation*}
        \biggl( \fint_E |f(x)|^p \dd x \biggr)^{\frac{1}{p}}
        \ge \lim_{p \to 0} \biggl( \fint_E |f(x)|^p \dd x \biggr)^{\frac{1}{p}}
        \ge \Bigl(\frac{C |E|}{eR (1+R)^{d-1}}\Bigr)^n \Phi(0).
    \end{equation*}
    We conclude with $\Phi(0) = \|f\|_{L^\infty(B_R)}$.
\end{proof}

\begin{lem}
    \label{lem::FM}
    Let $Q = \prod_{j=1}^d I_j$ where $I_j \subset \mathbb{R}$ are nonempty open intervals, and let $\Lambda = \prod_{j=1}^d Z_j$ where $Z_j$ are nonempty finite subsets of $\mathbb{R}$.
    Then for any trigonometric polynomial
    $f(x) = \sum_{k \in \Lambda} c_k e^{ik\cdot x}$
    and for all Borel $E \subset Q$, there holds the estimate
    \begin{equation}
        \label{eq::FM-ineq}
        \|f\|_{L^\infty(E)} \ge \Bigl(\frac{C |E|}{|Q|}\Bigr)^{\sum_{j=1}^d (|Z_j|-1)} \|f\|_{L^\infty(Q)}.
    \end{equation}
    
    In addition, if $E = \prod_{j=1}^d J_j$ where $J_j$ are Borel subsets of $I_j$, then for all $p\in(0,\infty]$,
    \begin{equation}
        \label{eq::FM-finer-ineq}
        \biggl( \fint_E |f(x)|^p \biggr)^{\frac{1}{p}} \ge \prod_{1\le j \le d} \biggl(\frac{C|J_j|}{|I_j|}\Bigr)^{|Z_j|-1} \biggl( \fint_Q |f(x)|^p \Bigr)^{\frac{1}{p}}.
    \end{equation}
\end{lem}
\begin{proof}
    The inequality \eqref{eq::FM-ineq} was proven by Fontes-Merz \cite{Fontes-Merz} when $Z_j \subset \mathbb{Z}$.
    The same proof clearly applies to the general situation.
    So we shall omit its proof here.

    To obtain \eqref{eq::FM-finer-ineq}, we will prove by induction that the following estimate holds for all $\nu=1,\ldots,d$ and (if $\nu>1$) for all $x_j \in I_j : j \le \nu-1$,
    \begin{equation}
        \label{eq::FM-ineq-induction}
        \biggl( \fint_{\prod_{j=\nu}^d J_j} |f(x)|^p \dd x_{\nu\le j\le d} \biggr)^{\frac{1}{p}}
        \ge \prod_{\nu\le j \le d} \Bigl(\frac{C|J_j|}{|I_j|}\Bigr)^{|Z_j|-1}  \biggl( \fint_{\prod_{j=\nu}^d I_j} |f(x)|^p \dd x_{\nu\le j\le d} \biggr)^{\frac{1}{p}},
    \end{equation}
    since this implies \eqref{eq::FM-finer-ineq} by taking $\nu=1$.
    Indeed, when $\nu=d$, this follows directly from \eqref{eq::nazarov-infty}.
    Suppose that \eqref{eq::FM-ineq-induction} holds for $\nu>1$, then for all $x_j \in I_j : j \le \nu-1$,
    \begin{align*}
        \biggl( \fint_{\prod_{j=\nu-1}^d J_j} |f(x)|^p &\dd x_{\nu-1\le j\le d} \biggr)^{\frac{1}{p}}
        = \biggl( \fint_{J_\nu} \fint_{\prod_{j=\nu}^d J_j} |f(x)|^p \dd x_{\nu\le j\le d} \dd x_{\nu-1} \biggr)^{\frac{1}{p}} \\
        & \ge \prod_{\nu\le j \le d} \Bigl(\frac{C|J_j|}{|I_j|}\Bigr)^{|Z_j|-1}  \biggl( \fint_{J_\nu} \fint_{\prod_{j=\nu}^d I_j} |f(x)|^p \dd x_{\nu\le j\le d} \dd x_{\nu-1} \biggr)^{\frac{1}{p}} \\
        & \ge \prod_{\nu-1\le j \le d} \Bigl(\frac{C|J_j|}{|I_j|}\Bigr)^{|Z_j|-1}  \biggl( \fint_{\prod_{j=\nu-1}^d I_j} |f(x)|^p \dd x_{\nu-1\le j\le d} \biggr)^{\frac{1}{p}}.
        \qedhere
    \end{align*}
\end{proof}

\section{Relations with quantum limits}
\label{sec::QL-relation}

We prove \eqref{eq::QL-relation}, which follows directly from the following lemma.

\begin{lem}
    If $G \subset \mathbb{T}^d$ is an open subset with $|\partial G| = 0$, then
    \begin{equation*}
        \inf_{\mu \in \mathcal{Q}} \mu(G)
        = \liminf_{\lambda\to\infty} \inf_{v \in \mathcal{E}_d(\lambda)^*} \int_{G} |u|^2 \dd x,
    \end{equation*}
    where $\mathcal{E}_d(\lambda)^*$ is the set of all $u \in \mathcal{E}_d(\lambda)$ with $\|u\|_{L^2(\mathbb{T}^d)} = 1$.
\end{lem}
\begin{proof}
    By Bourgain (c.f.~\cite{Jakobson}), every $\mu \in \mathcal{Q}$ is absolutely continuous with respect to the Lebesgue measure on $\mathbb{T}^d$, and thus $\mu(\partial G) = 0$.
    So, if $|u_n|^2 \dd x \rightharpoonup \mu$ where $u_n \in \mathcal{E}_d(\lambda_n)^*$ with $\lambda_n\to\infty$, then
    \begin{equation}
    \label{eq::QL-key-identity}
        \lim_{n \to \infty} \int_{G} |u_n|^2 \dd x = \mu(G).
    \end{equation}
    Taking first the infimum among all such sequences $u_n$ in \eqref{eq::QL-key-identity} and then $\inf_{\mu \in \mathcal{Q}}$ yields
    \begin{equation*}
        \liminf_{\lambda\to\infty} \inf_{u \in \mathcal{E}_d(\lambda)^*} \int_{G} |u|^2 \dd x
        \le \inf_{\mu \in \mathcal{Q}} \mu(G).
    \end{equation*}

    Next, for all $\epsilon > 0$ and eigenvalue $\lambda$, choose $u_\lambda \in \mathcal{E}_d(\lambda)$ such that
    \begin{equation*}
        \inf_{u \in \mathcal{E}_d(\lambda)^*} \int_{G} |u|^2 \dd x + \epsilon
        \ge \int_G |u_\lambda|^2 \dd x.
    \end{equation*}
    Taking $\liminf_{\lambda\to\infty}$, using \eqref{eq::QL-key-identity} and the fact that any sequence of $u_\lambda \in \mathcal{E}_d(\lambda)^*$ contains a subsequence, say $u_{\lambda_n}$, such that $|u_{\lambda_n}|^2 \dd x$ weakly converges, we obtain the estimate
    \begin{equation*}
        \liminf_{\lambda\to\infty} \inf_{u \in \mathcal{E}_d(\lambda)^*} \int_{G} |u|^2 \dd x + \epsilon 
        \ge \liminf_{\lambda \to \infty} \int_G |u_\lambda|^2 \dd x
        \ge \inf_{\mu \in \mathcal{Q}} \mu(G).
    \end{equation*}
    We conclude by taking $\epsilon \to 0$.
\end{proof}

\section{Cluster structure of integer points on spheres}
\label{sec::cluster}

We review some results on the cluster structure of integer points on spheres.

\subsection{Cluster structure on circles}

\begin{lem}[Jarník \cite{Jarnik1926}]
    \label{lem::jarnik}
    Let $\lambda > 0$.
    On the circle $\lambda\mathbb{S}$, any arc of length $<\sqrt{2}\lambda^{\frac{1}{3}}$ contains at most two integer points, i.e., points in $\mathbb{Z}^2$.
\end{lem}

\begin{lem}[Cilleruelo--Cordoba \cite{CillerueloCordoba}]
    \label{lem::CC}
    Let $\lambda>0$, $m\in \mathbb{N}\backslash \{0\}$ and $\delta(m) = \bigl(4\bigl\lfloor\frac{m}{2}\bigr\rfloor+2\bigr)^{-1}$.
    On the circle $\lambda\mathbb{S}$, any arc of length $< \sqrt{2} \lambda^{\frac{1}{2}-\delta(m)}$ contains at most $m$ integer points.
\end{lem}

\subsection{Cluster structure on high dimensional spheres and ellipsoids}

\begin{df}
    If $E$ is an ellipsoid, then we call the radius of $E$ the length of its longest principal axes and call the ratio of $E$ the maximum quotient between its principal axes.
\end{df}

We refer to the Connes~\cite[Lem.~1]{Connes} for the following lemmas.

\begin{lem}
\label{lem::jarnik-connes}
    For all $d \ge 2$ and $\nu \ge 1$, there exists $c_{d,\nu}>0$ such that, for any ellipsoid $E$ with radius $\lambda$ and ratio $\le \nu$, and for all $k_0 \in E$, we have:
    \begin{itemize}
        \item The following set lies on an affine hyperplane:
        \begin{equation*}
            H(k_0) \coloneq \{k \in E \cap \mathbb{Z}^d : |k-k_0| < c_{d,\nu} \lambda^{\frac{1}{d+1}}\}.
        \end{equation*}
        \item If $H(k_0)$ determines a hyperplane, then
        \begin{equation*}
            \dist\bigl( (E \cap \mathbb{Z}^d) \backslash H(k_0),H(k_0) \bigr) \ge c_{d,\nu} \lambda^{\frac{1}{d+1}}.
        \end{equation*}
    \end{itemize}
\end{lem}

\begin{lem}
    \label{lem::connes}
    For all $d \ge 2$ and $\nu \ge 1$, there exists a function $\tau_{d,\nu}$ satisfying
    \begin{equation*}
        \tau_{d,\nu}(\lambda) \asymp \lambda^{\frac{2}{(d+1)!}}
    \end{equation*}
    as $\lambda\to\infty$, such that, if $E\subset\mathbb{R}^d$ is an ellipsoid with radius $\lambda$ and ratio $\le \nu$, and if $\Omega = (\Omega_\alpha)_\alpha$ the partition of $E\cap\mathbb{Z}^d$ given by the graph $G(E\cap\mathbb{Z}^d,\tau_{d,\nu}(\lambda))$ as according to Definition~\ref{def::graph}, then $\Omega_\alpha$ lies on an affine hyperplane for all $\alpha$.
\end{lem}

\bibliographystyle{plain}
\bibliography{references}

\end{document}